\documentclass[11pt]{amsart}
\usepackage{amsmath}
\usepackage{amsfonts}
\usepackage{amssymb}
\usepackage{graphicx}
\usepackage{float, verbatim}
\usepackage{enumerate}
\usepackage{color}

\newcommand{\lld}{\underline{\dim}_{\mathrm{loc}}}

\renewcommand{\geq}{\geqslant}
\renewcommand{\leq}{\leqslant}
\renewcommand{\epsilon}{\varepsilon}

\newtheorem{thm}{Theorem}[section]
\newtheorem{lma}[thm]{Lemma}
\newtheorem{cor}[thm]{Corollary}

\newtheorem{prop}[thm]{Proposition}

\newtheorem{ques}[thm]{Question}

\numberwithin{equation}{section}

\begin{document}

\title[Assouad  spectra]{Assouad type spectra for some fractal families}


\author[J. M. Fraser]{Jonathan M. Fraser}
\address{Mathematical Institute\\ The University of St Andrews\\ St Andrews\\ KY16 9SS\\ Scotland }
\urladdr{http://www.mcs.st-and.ac.uk/~jmf32/}
\curraddr{}
\email{jmf32@st-andrews.ac.uk}
\thanks{The work of JMF was partially supported by a \emph{Leverhulme Trust Research Fellowship}.  This work began  while both authors were at the University of Manchester and they are grateful for the inspiring atmosphere they enjoyed during their time there.  They are grateful to Thomas Jordan for helpful discussions.}

\author[H. Yu]{Han Yu}
\address{Mathematical Institute\\ The University of St Andrews\\ St Andrews\\ KY16 9SS\\ Scotland }
\curraddr{}
\email{hy25@st-andrews.ac.uk}
\thanks{}

\date{}



\begin{abstract} In a previous paper we introduced a new `dimension spectrum', motivated by the Assouad dimension,  designed to give precise information about the scaling structure and homogeneity of a metric space.   In this paper we compute the spectrum explicitly for a range of well-studied fractal sets, including: the self-affine carpets of Bedford and McMullen, self-similar and self-conformal sets with overlaps, Mandelbrot percolation, and Moran constructions.  We find that the spectrum behaves differently for each of these models and can take on a rich variety of forms.  We also consider some applications, including the provision of new bi-Lipschitz invariants and bounds on a family of `tail densities' defined for subsets of the integers.
\end{abstract}

\maketitle

\section{Assouad type spectra and organisation of paper}

The Assouad dimension is an important concept in geometric measure theory and analysis.  Roughly speaking it quantifies the `thickest' part of a given metric space, and appears naturally in a wide variety of contexts, such as embedding theory.  Indeed, proving his celebrated embedding theorem was Assouad's main motivation for studying and popularising the concept during his PhD \cite{assouadphd}.    Generally, we will be interested in `how big' a set $F$ is at a small scale $r>0$ (as $r \to 0$).  This is made precise by the covering function   $N(F,r)$ which is defined to be  the minimal number of $r$-balls required to cover $F$.   For example, the (upper) box dimension   extracts the smallest exponent $\alpha$ such that  $N(F,r) \leq  (1/r)^\alpha$ as $r \to 0$. The Assouad dimension, on the other hand,  considers two scales $R>r>0$.  The large scale $R$ focuses  ones attention on a small $R$-ball $B(x,R)$ centered in $E$ and then one    extracts the smallest exponent $\alpha$ such that $N(B(x,R) \cap F,r) \leq C (R/r)^{\alpha}$ for some constant $C$ and for all centres  and pairs of scales.  In \cite{Spectra16} we proposed the following modification of the Assouad dimension where one \emph{fixes} the relationship between the two scales.   More precisely, we define the \emph{Assouad spectrum} of a nonempty set $F$ in a metric space to be the function
\begin{eqnarray*}
\theta  \ \mapsto \ \dim_{\mathrm{A}}^\theta F &=& \inf \bigg\{ \alpha \  : \   (\exists C>0) \, (\exists \rho>0) \, (\forall 0<R\leq \rho) \,  (\forall x \in F) \\ \\
&\,& \qquad \qquad  \qquad \qquad   N \big( B(x,R) \cap F ,R^{1/\theta} \big) \ \leq \ C \left(\frac{R}{R^{1/\theta}}\right)^\alpha \bigg\}
\end{eqnarray*}
where  $\theta$ varies over  $(0,1)$.  If  one replaced $R^{1/\theta}$ by a `free scale' $r \in (0,R)$, then one obtains the familiar Assouad dimension, which we denote by $\dim_{\mathrm{A}}F$.  We also consider a `dual spectrum', which is motivated by the lower dimension.  The  \emph{lower spectrum} is given by the function
\begin{eqnarray*}
\theta  \ \mapsto \  \dim_\text{L}^\theta F &=&   \sup \bigg\{ \alpha \  : \   (\exists C>0) \, (\exists \rho>0) \, (\forall 0<R\leq \rho) \,  (\forall x \in F) \\ \\
&\,& \qquad \qquad  \qquad \qquad   N \big( B(x,R) \cap F,R^{1/\theta} \big) \ \geq \ C \left(\frac{R}{R^{1/\theta}}\right)^\alpha \bigg\}
\end{eqnarray*}
where  $\theta$ varies over  $(0,1)$.  Again, if  one replaced $R^{1/\theta}$ by a `free scale' $r \in (0,R)$, then one obtains the familiar lower dimension, which we denote by $\dim_{\mathrm{L}}F$.   We refer the reader to  \cite{Spectra16} for a detailed discussion of the basic properties of the Assouad and lower spectra.  We will sometimes refer to specific results from \cite{Spectra16}  when we require them, but we summarise the most important properties in the following section.

The key motivation behind these Assouad type spectra is that they provide more detailed and precise information about the scaling structure of the space.  In particular, one obtains a spectrum of exponents, rather than a single one.    In this paper we will explicitly compute the spectra in a variety of important contexts. What we find is that the spectra display a wide variety of different features and forms, reflecting the differences between the  models we consider.  Specifically, our main results are contained in four virtually stand alone sections where we consider: self-affine carpets (Section \ref{CarpetsSect}), self-similar and self-conformal sets with overlaps (Section \ref{Self-similarSect}), Mandelbrot percolation (Section \ref{MandelbrotSect}), and Moran constructions (Section \ref{MoranSection}).

\section{Notation and preliminaries}

We begin with a summary of some notation we will use throughout the paper. For real-valued functions $f,g$, we write $f(x)\lesssim g(x)$ to mean    that there exists a universal constant $M>0$,    independent of $x$, such that $f(x)\leq Mg(x)$. Some readers may be more familiar with the notation $f(x) = O(g(x))$, which is sometimes more convenient and means the same thing.   Similarly, $f(x)\gtrsim g(x)$    means that $f(x)\geq Mg(x)$ with a universal constant independent of $x$.  If both $f(x)\lesssim g(x)$ and $f(x)\gtrsim g(x)$, then we write $f(x)\asymp g(x)$.  Generally one should think of $x$ as being the tuple consisting of all variables in    the expression `$f(x)$'.

For a real number $a$,    we write $a^+$ to denote    a real number that is strictly larger than $a$ but can be chosen as close to $a$ as we wish. Similarly, we write $a^-$ to denote a real number that is strictly less than $a$ but can be chosen as close to $a$ as we wish.  For real numbers $a,b$, we write $a \wedge b$ for the minimum of the two numbers    and $a \vee b$ for the maximum.  Also, for a non-negative real number $x \geq 0$, we write $[x]$ for the integer part of $x$.

The Assouad and lower spectra are closely related to the upper and lower box dimension.  The \emph{upper} and \emph{lower box dimensions} of a totally bounded set $F$ are defined by
\[
\overline{\dim}_{\mathrm{B}}F  \ = \  \limsup_{R \to 0} \   \frac{\log  N \big( F, R\big) }{-\log R} \qquad \text{ and } \qquad   \underline{\dim}_\mathrm{B} F  \ = \  \liminf_{R \to 0} \   \frac{\log  N \big( F, R\big) }{-\log R}
\]
and if the upper and lower box dimensions coincide we call the common value the \emph{box dimension} of $F$ and denote it by $\dim_\mathrm{B} F$.   Note that, similar to the Assouad and lower spectra, the covering function $N(\cdot, R)$ can be taken to be numerous related covering or packing functions and it is sometimes convenient to use a variety of different definitions in practice.  In particular we note the following general relationships which hold for any totally bounded set $F$:
\[
 \dim_\text{L} F \      \leq   \   \underline{\dim}_\text{B} F   \   \leq   \  \overline{\dim}_\text{B} F  \   \leq   \ \dim_\text{A} F.
\]
For more detailed information on dimension theory and the basic properties of the various dimensions we consider here, we refer the reader to \cite{Fraser, Luukkainen, Falconer, Robinson}.  In \cite{Spectra16} we conducted a thorough investigation of the analytic and geometric properties of the Assouad and lower spectra.  Some of the key properties we will use in this paper are as follows:
\begin{enumerate}
\item \cite[Propositions 3.1  and 3.10]{Spectra16}.   For any totally bounded $F$ and $\theta \in (0,1)$ one has the following general bounds:
\[
 \dim_\text{L} F \      \leq   \  \dim_\text{L}^\theta F \   \leq   \  \underline{\dim}_\text{B} F  
\]
and
\[
\overline{\dim}_\text{B} F  \    \leq   \  \dim_\text{A}^\theta F   \   \leq   \   \frac{\overline{\dim}_\text{B} F}{1-\theta}   \wedge  \dim_\text{A} F.
\]
\item \cite[Corollaries 3.5 and 3.11]{Spectra16}.   The Assouad and lower spectra are continuous in $\theta$.
\item \cite[Corollaries 4.9 and 4.10]{Spectra16}. The Assouad and lower spectra are invariant under bi-Lipschitz and quasi-bi-Lipschitz maps. 
\end{enumerate}

\section{Self-affine carpets}  \label{CarpetsSect}

Self-affine sets are an important class of fractal subsets of Euclidean space.  They are attractors of finite iterated function systems consisting of affine contractions.  Due to the fact that affine maps can scale by different amounts in different directions, the dimension theory of self-affine sets is much more complicated than their self-similar counter parts; where the defining maps are similarities.  Self-affine sets  have been studied intensively  for over 30 years and have important connections with many other fields, such as non-conformal dynamical systems.  An  important contribution to this area was Bedford and McMullen's independent treatments of a restricted class of self-affine sets \cite{Bedford, McMullen}.  These `Bedford-McMullen carpets' are planar self-affine sets where the affine maps all have the same linear part; which is a simple diagonal matrix.  These sets are much simpler than general self-affine sets, but still display many of the key features: the maps scale by different amounts in the two principle directions and so the constituent pieces become increasingly elongated as one iterates the construction.  Due to their simple and explicit form, the dimension theory of Bedford-McMullen carpets is rather well-developed.  Indeed, Bedford and McMullen computed their Hausdorff and box dimensions in the mid-1980s, Mackay computed their Assouad dimensions in 2011 \cite{Mackay}, and Fraser computed their lower dimensions in 2014 \cite{Fraser}.  In this section we give explicit formulae for the Assouad and lower spectra.

We begin by recalling the construction. Fix $m,n \in \mathbb{N}$ with $2\leq m<n$ and divide the unit square $[0,1]^2$ into an $m \times n$  regular grid.  Let $\mathcal{I} = \{0, \dots, m-1\}$ and $\mathcal{J} = \{0,\dots, n-1\}$ and label the $mn$ rectangles in the regular grid by $\mathcal{I} \times \mathcal{J}$ counting from bottom left to top right.  Choose a subset of the rectangles $\mathcal{D} \subseteq \mathcal{I} \times \mathcal{J}$ of size at least 2 and for each $d = (i,j) \in \mathcal{D}$, associate a contraction $S_d: [0,1]^2 \to [0,1]^2$ defined by
\[
S_d(x,y) = (x/m+i/m, \,  y/n+j/n).
\]
Let
\[
\mathcal{D}^\infty = \{ \mathbf{d} = (d_1, d_2, \dots) : d_l = (i_l,j_l) \in \mathcal{D} \}
\]
be the set of infinite words over $\mathcal{D}$ and let $\Pi :  \mathcal{D}^\infty \to [0,1]^2$ be the canonical map from the symbolic space to the geometric space defined by
\[
\Pi(\mathbf{d}) \ =  \ \bigcap_{l \in \mathbb{N}} \, S_{d_1} \circ \cdots \circ  S_{d_l} \big([0,1]^2\big).
\]
The set $F = \Pi(\mathcal{D}^\infty)$ is the Bedford-McMullen carpet.  We thus obtain a symbolic coding of the set $F$, but we note that some points may not have a unique code. The set $F$ can be viewed in several other contexts.  For example, $F$ is the attractor of the \emph{iterated function system} (IFS) $\{S_d\}_{d \in \mathcal{D}}$, i.e. it is the unique non-empty compact set satisfying
\[
F \ = \ \bigcup_{d \in \mathcal{D}} S_d(F).
\]
In the context of hyperbolic dynamical systems, Bedford-McMullen carpets are also of particular interest.  Identifying $[0,1]^2$ with the 2-torus in the natural way, $F$ is a (forward and backward) invariant repeller for the hyperbolic toral endomorphism 
\[
(x,y) \to (mx \hspace{-2mm} \mod 1, \, ny\hspace{-2mm}  \mod 1).
\] 

\begin{figure}[H] 
	\centering
	\includegraphics[width=125mm]{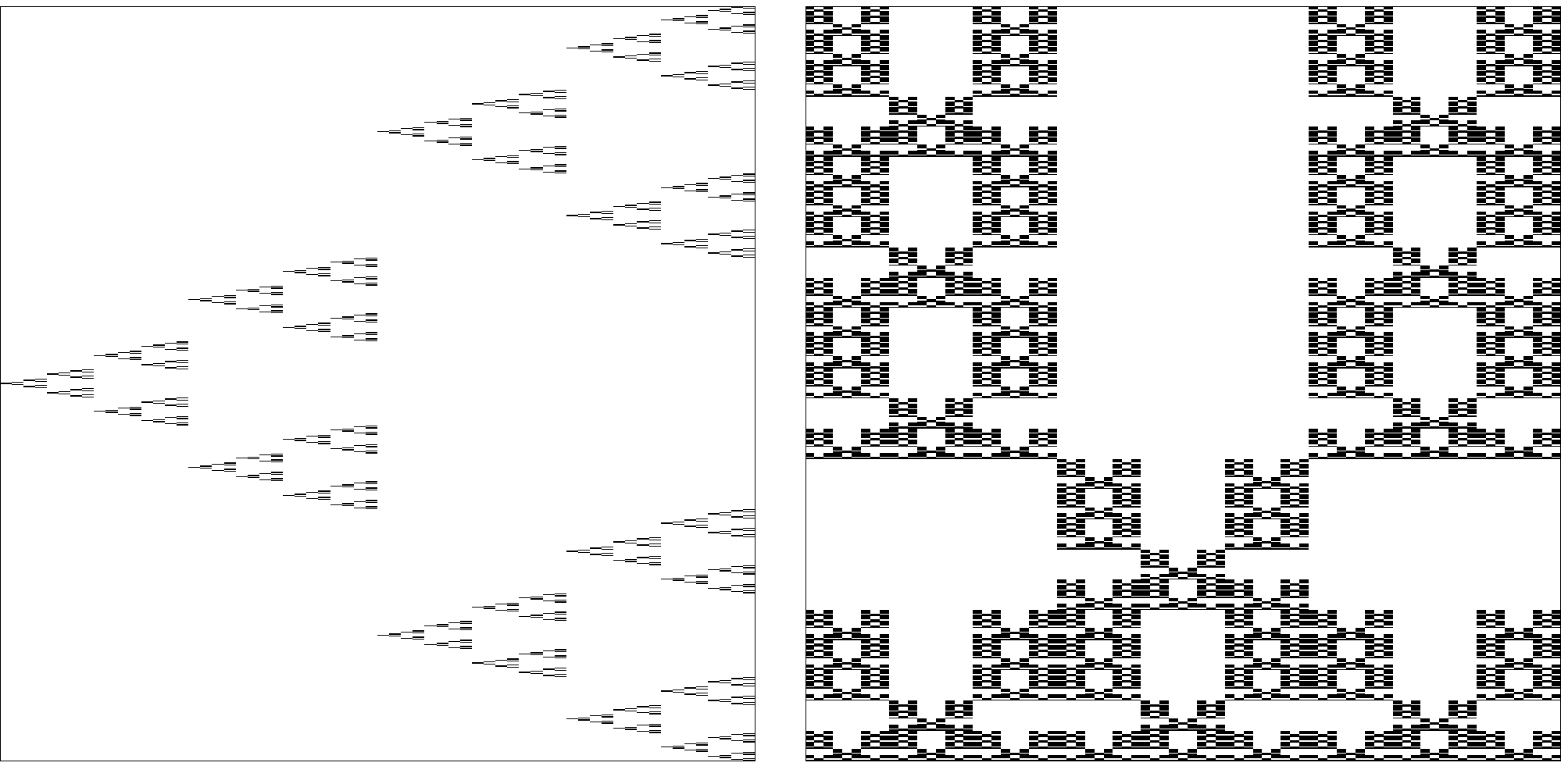}
\caption{Two examples of self-affine carpets of the type introduced by Bedford and McMullen.  On the left $m=2$ and $n=3$ and on the right $m=3$ and $n=5$.\label{carpets}}
\end{figure}

In order to state the dimension results mentioned above we need some more notation.   Let $\pi : \mathcal{I} \times \mathcal{J} \to \mathcal{I}$ be the projection onto the first coordinate.  For $i \in \mathcal{I}$, let
\[
C_i  \ =  \ \lvert \{(i',j') \in \mathcal{D} : i'=i\} \rvert  \ = \  \lvert \pi^{-1}(i) \cap  \mathcal{D} \rvert
\]
be the number of chosen rectangles in the $i$th column, let $C_{\max} = \max_{i \in \mathcal{I}} C_i$ and $C_{\min} = \min_{i \in \pi \mathcal{D}} C_i$.
\begin{thm}[Bedford-McMullen-Mackay-Fraser] \label{carpetdimensions}
The Assouad,  box, Hausdorff and lower dimensions of a Bedford-McMullen carpet $F$ are given by 
\[
\dim_\text{\emph{A}} F \ =  \  \frac{\log \lvert\pi\mathcal{D} \rvert }{ \log m} \, + \,  \frac{ \log C_{\max}}{ \log n},
\]
\[
\dim_\text{\emph{B}} F \ =  \  \frac{\log \lvert\pi\mathcal{D} \rvert }{ \log m} \, + \,  \frac{ \log(\lvert\mathcal{D} \rvert/ \lvert \pi\mathcal{D} \rvert )}{ \log n}
\]
\[
\dim_\mathrm{H} F \ = \ \frac{\log \sum_{i \in \mathcal{I}} C_i^{\log m/\log n}}{\log m}
\]
and
\[
\dim_\text{\emph{L}} F \ =  \  \frac{\log \lvert\pi\mathcal{D} \rvert }{ \log m} \, + \,  \frac{ \log C_{\min}}{ \log n}.
\]
\end{thm}

Observe that the term $\log \lvert\pi\mathcal{D} \rvert /\log m$ appears in three of the dimension formulae.  This is the dimension of the self-similar projection of $F$ onto the first coordinate. Where this term appears, the second term relates to the dimension of fibres: for the Assouad dimension, $\log C_{\max}/ \log n$ is the maximal fibre dimension (also given by a self-similar set), for the lower dimension, $\log C_{\min}/ \log n$ is the minimal fibre dimension and, for the box dimension, the second term can be interpreted as an `average fibre dimension', observing that $\lvert\mathcal{D} \rvert/ \lvert \pi\mathcal{D} \rvert$ is the arithmetic average of the $C_i$.  The box,  Hausdorff, Assouad and lower dimensions coincide if and only if $C_i$ is the same for all $i \in \pi \mathcal{D}$, commonly referred to as the \emph{uniform fibres case}, and otherwise they are all distinct, see \cite{Mackay, Fraser}.

Before we discuss our main results we provide a minor addition to Theorem \ref{carpetdimensions}.   The lower dimension is well-known to have many strange properties which may not be seen as desirable for a `dimension' to satisfy.  For example, it is not monotone  and it may take the value 0 for sets with non-empty interior, see \cite[Example 2.5]{Fraser}.  One can modify the definition to get rid of these (perhaps) strange properties by defining the \emph{modified lower dimension} by
\[
\dim_\text{ML} F \ = \  \sup \left\{ \dim_\text{L} E \  : \  \emptyset \neq E \subseteq F \right\}.
\]
We mention this here simply to point out that this can be computed for self-affine carpets as a direct consequence of results of Ferguson-Jordan-Shmerkin \cite{fjs} and the formula for the lower dimension given by Fraser \cite{Fraser}.  Perhaps surprisingly the modified lower dimension is equal to the \emph{Hausdorff} dimension and not the lower dimension.

\begin{cor} \label{modifiedlower}
The modified lower dimension of a Bedford-McMullen carpet $F$ is given by
\[
\dim_\mathrm{ML} F \ = \ \dim_\mathrm{H} F \ = \ \frac{\log \sum_{i \in \mathcal{I}} C_i^{\log m/\log n}}{\log m}.
\]
Moreover, it is also true that if $F$ is a self-affine carpet of the type considered by Bara\'nski \cite{baranski} or Lalley-Gatzouras \cite{lalley-gatz}, then
\[
\dim_\mathrm{ML} F \ = \ \dim_\mathrm{H} F.
\]
We refer the readers to \cite{baranski, lalley-gatz} for the precise formulae.
\end{cor}
\begin{proof}
Since the modified lower dimension of a compact set is always bounded above by the Hausdorff dimension we need only prove the lower bound.  It suffices to prove that for any $\varepsilon>0$ we can find a subset of $F$ which has lower dimension within $\varepsilon$ of the Hausdorff dimension.  Via an elegant application of Stirling's formula, Ferguson-Jordan-Shmerkin  \cite[Lemma 4.3]{fjs} showed that one can always find a subset of such a carpet generated by a subsystem of an iterate of the original IFS which both has Hausdorff dimension within $\varepsilon$ of the original  Hausdorff dimension and has the key additional property that the fibres are uniform.  Then it follows from \cite[Theorem 2.13]{Fraser} that the lower dimension and the Hausdorff dimension of the subset coincide, thus proving the result.

The result for more general carpets follows in exactly the same way and we omit the details.
\end{proof}

Our main result on self-affine carpets is the following theorem, which gives explicit expressions for the Assouad and lower spectra.

\begin{thm} \label{CarpetsSpectra}
Let $F$ be a Bedford-McMullen carpet.  Then for all $\theta \in (0,1)$ we have
\[
\dim_{\mathrm{A}}^\theta F \ = \   \frac{   \dim_\mathrm{B} F \ -  \ \theta \, \bigg( \frac{\log (\lvert \mathcal{D} \rvert/C_{\max})}{\log m} \, + \,  \frac{\log C_{\max}}{\log n}\bigg)  }{1- \theta} \quad \wedge \ \dim_\mathrm{A} F
\]
and
\[
\dim_\mathrm{L}^\theta  F \ = \   \frac{   \dim_\mathrm{B} F \ -  \ \theta \,  \bigg( \frac{\log (\lvert \mathcal{D} \rvert/C_{\min})}{\log m} \, + \,  \frac{\log C_{\min}}{\log n}\bigg)  }{1- \theta} \quad \vee  \  \dim_\mathrm{L} F.
\]
Moreover, both both spectra have only one phase transition which occurs at $\theta =\log m/\log n$ and both are analytic and monotonic in the interval $(0,\log m/\log n)$ and constant in $(\log m / \log n, 1)$.
\end{thm}

An immediate consequence of Theorem \ref{CarpetsSpectra} is that for any Bedford-McMullen carpet with non-uniform fibres the Assouad spectrum is \emph{strictly smaller} than the general upper bound given by \cite[Proposition 3.1]{Spectra16} for all $\theta \in (0,\log m/\log n)$.  Also, the unique phase transition always occurs to the right of the phase transition in the general bound, i.e. 
\[
\frac{\log m}{\log n} > 1- \frac{\dim_\mathrm{B} F}{\dim_\mathrm{A} F}.
\]
These observations follow by simple manipulation of the various dimension formulae.

\begin{figure}[H]
	\centering
	\includegraphics[width=145mm]{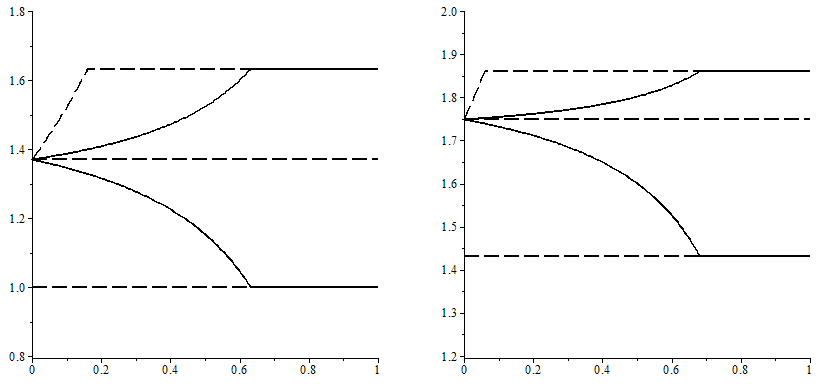}
\caption{Plots of the Assouad and lower spectra for the two carpets depicted in Figure \ref{carpets}.  On the left: $m=2$, $n=3$, $\lvert \mathcal{D} \rvert  = 3$,  $\lvert \pi \mathcal{D}  \rvert = 2$, $C_{\max}=2$ and  $C_{\min} = 1$. On the right: $m=3$, $n=5$, $\lvert \mathcal{D} \rvert = 10$,  $\lvert \pi \mathcal{D}  \rvert = 3$, $C_{\max}=4$ and  $C_{\min} = 2$.  The general bounds from \cite[Proposition 3.1]{Spectra16} and \cite[Proposition 3.10]{Spectra16} are shown with dashed lines. \label{carpetsspectrum}}
\end{figure}

A, perhaps surprising, corollary of Theorem \ref{CarpetsSpectra}  is that the lower spectra is not uniformly bounded above by the Hausdorff or modified lower dimension.  This might be surprising because the lower spectrum is based on the lower dimension, which \emph{is} bounded above by the Hausdorff and  modified lower dimension for compact sets.  The Assouad spectrum does not violate any of the general inequalities between the Assouad dimension and the other common dimensions since it is uniformly bounded below by the upper box dimension, see \cite[Proposition 3.1]{Spectra16}.  In fact it follows that for any Bedford-McMullen carpet with non-uniform fibres there will be an open interval $(0, \theta_0)$ within which the lower spectrum is strictly larger than the Hausdorff and  modified lower dimension.  This follows since for such carpets the lower spectrum is continuous and converges to the box dimension as $\theta \to 0$ which is strictly larger than the Hausdorff and  modified lower dimension.  This goes someway to justifying the sharpness of \cite[Proposition 3.10]{Spectra16}.

Theorem \ref{CarpetsSpectra} proves that the ratio $\frac{\log m}{\log n}$ is a bi-Lipschitz invariant within the class of Bedford-McMullen carpets.  To the best of our knowledge this has not been observed before. Since sets are bi-Lipschitz equivalent to themselves this invariant can be useful in determining what kinds of affine IFSs give rise to a particular carpet.  For example, let $F$ be a Bedford-McMullen carpet generated by an IFS using an $m \times n$ grid for some $2 \leq m<n$.  Assume that $F$ does not have uniform fibres and that it can also be viewed as a Bedford-McMullen carpet generated by an IFS using an $m' \times n'$ grid for some $2 \leq m'<n'$.  Then we may conclude that
\[
\frac{\log m}{\log n} = \frac{\log m'}{\log n'}.
\]
The question of when two Bedford-McMullen carpets are bi-Lipschitz equivalent has been considered by Li-Li-Miao \cite{Miao}, where they give positive results assuming some natural conditions, including that $m,n$ are constant.  In some sense, our results compliment these results as we provide a mechanism for proving negative results.  We demonstrate this by the following example.

\begin{prop}\label{T2}
We can find two topologically equivalent Bedford-McMullen carpets with the same  box counting,  lower and Assouad dimensions, but different spectra.  Therefore $E$ and $F$ are not bi-Lipschitz invariant despite this not being revealed by knowledge of the other dimensions.
\end{prop}

\begin{proof}

We construct the first Bedford-McMullen carpet  such that $m=5$, $n=6$, $\lvert \pi \mathcal D \rvert = 3$, $\lvert \mathcal{D} \rvert = 6$, $C_{\max} = 3$, $C_{\min} = 1$ and the second carpet such that $m=5$, $n=36$, $\lvert \pi \mathcal D \rvert= 3$, $\lvert \mathcal{D} \rvert  = 12$, $C_{\max} = 9$, $C_{\min} = 1$.  Moreover, we may choose the rectangles satisfying these parameters so that they do not touch each other, rendering both carpets totally disconnected and thus homeomorphic.  It follows from Theorem \ref{carpetdimensions} that for both carpets
\[
\dim_\text{A} F \ =  \  \frac{\log 3 }{ \log 5} \, + \,  \frac{ \log 3}{ \log 6},
\]
\[
\dim_\text{B} F \ =  \  \frac{\log 3 }{ \log 5} \, + \,  \frac{ \log 2}{ \log 6}
\]
and
\[
\dim_\text{L} F \ =  \  \frac{\log 3 }{ \log 5} ,
\]
but, since the ratio $\log m/\log n$ is not the same for both constructions, they necessarily have different spectra.
\end{proof}

For the above examples, we guaranteed the spectra would be different by ensuring that $\frac{\log m}{\log n}\neq \frac{\log m'}{\log n'}$.  This turns out to be necessary in finding such examples.  The following corollary expresses this precisely and follows immediately by rearranging the formulae given in Theorems \ref{carpetdimensions} and \ref{CarpetsSpectra}.

\begin{cor}
The dimension spectra of Bedford-McMullen carpets are completely determined by the ratio $\log m/ \log n$ and the  box, lower and Assouad dimensions of the carpet.  More precisely, we have
\[
\dim_{\mathrm{A}}^\theta F \ = \   \frac{   \dim_\mathrm{B} F \ -  \ \theta \, \left(\dim_\mathrm{A} F  - \left( \dim_\mathrm{A} F - \dim_\mathrm{B} F\right) \frac{\log n}{\log m} \right) }{1- \theta} \ \wedge \ \dim_\mathrm{A} F
\]
and
\[
\dim_\mathrm{L}^\theta  F \ = \  \frac{   \dim_\mathrm{B} F \ -  \ \theta \, \left(\dim_\mathrm{L} F  +\left( \dim_\mathrm{B} F - \dim_\mathrm{L} F\right) \frac{\log n}{\log m} \right) }{1- \theta} \ \vee \ \dim_\mathrm{L} F.
\]

\end{cor}

\subsection{Proof of Theorem \ref{CarpetsSpectra}}

Let $\mathbf{d} \in \mathcal{D}^\infty$ and $r>0$ be small.  Define $l_1(r), l_2(r)$ to be the unique natural numbers satisfying
\[
m^{-l_1(r)} \,  \leq \, r \, < \,  m^{-l_1(r)+1}
\]
and
\[
n^{-l_2(r)} \,  \leq \, r \, < \,  n^{-l_2(r)+1}.
\]
The approximate square `centred'  at $\mathbf{d} = ((i_1, j_1), (i_2, j_2), \dots)  \in \mathcal{D}^\infty$ with `radius' $r>0$ is defined by
\begin{eqnarray*}
Q(\mathbf{d}, r) &=& \Big\{ \mathbf{d}' = ((i'_1, j'_1), (i'_2, j'_2), \dots)  \in \mathcal{D}^\infty :   i'_l = i_l \ \forall l \leq l_1(r) \text{ and }  j'_l = j_l \ \forall  l \leq l_2(r)  \Big\},
\end{eqnarray*}
and the geometric projection of this set, $\Pi\big(Q(\mathbf{d}, r)\big)$, is a subset of $F$ which contains $\Pi(\mathbf{d})$ and naturally sits inside a rectangle which is `approximately a square' in that it has a base with length $m^{-l_1(r)} \in (r/m, r]$ and height $n^{-l_2(r)} \in (r/n, r]$.  One obtains equivalent definitions of the Assouad and lower spectra  if one replaces $B(x,R)$  by $\Pi\big(Q(\mathbf{d}, R)\big)$, i.e. we may use approximate squares instead of balls.

Our subsequent analysis breaks naturally into two cases, illustrated by the following lemma.

\begin{lma}
Let $R\in (0,1)$ and $\theta \in (0,1)$.  Then
\begin{enumerate}
\item if $\theta \leq \log m / \log n$, then
\[
 l_2(R) \ \leq \  l_1(R) \ \leq \  l_2\left(R^{1/\theta} \right) \ \leq \  l_1\left(R^{1/\theta} \right) 
\]
\item if $\theta \geq \log m / \log n$, then
\[
 l_2(R) \ \leq \ l_2\left(R^{1/\theta} \right) \ \leq \   l_1(R) \ \leq \  l_1\left(R^{1/\theta} \right) .
\]
\end{enumerate}
\end{lma}
\begin{proof}
This follows immediately from the definitions of $l_1(\cdot)$ and $l_2(\cdot)$.
\end{proof}

\begin{prop} \label{key1} Let $\theta  \in (0, \log m / \log n]$.  Then
\[
\dim_{\mathrm{A}}^\theta F  \ = \   \frac{   \bigg( \frac{\log \lvert\pi\mathcal{D} \rvert }{ \log m} \, + \,  \frac{ \log(\lvert\mathcal{D} \rvert/ \lvert \pi\mathcal{D} \rvert )}{ \log n}  \bigg) \ -  \ \theta \bigg( \frac{\log (\lvert \mathcal{D} \rvert/C_{\max})}{\log m} \, + \,  \frac{\log C_{\max}}{\log n}\bigg)  }{1- \theta}
\]
\end{prop}

\begin{proof}
Let
\[
\mathbf{d} = ((i_1, j_1), (i_2, j_2), \dots)  \in \mathcal{D}^\infty
\]
and $R \in (0,1)$.   When one looks at the approximate square $\Pi\big(Q(\mathbf{d}, R)$ one finds that it is made up of several horizontal strips with base length the same as that of the approximate square, i.e. $m^{-l_1(R)}$, and height $n^{-l_1(R)}$ which is considerably smaller, but still larger than $R^{1/\theta}/n$ since $ l_1(R) \ \leq \  l_2\left(R^{1/\theta} \right) $.  These strips are images of $F$ under $l_1(R)$-fold compositions of maps from $\{S_d\}_{d \in \mathcal{D}}$  and so to keep track of how many there are, one counts rectangles in the appropriate columns for each map in the composition.  The total number of such strips is seen to be:
\[
\prod_{l=l_2(R)+1}^{l_1(R)} C_{i_l}.
\]
We want to cover these strips by sets of diameter $R^{1/\theta}$ and so we iterate the construction inside each horizontal strip until the height is $n^{-l_2(R^{1/\theta})}$, which is approximately $R^{1/\theta}$. This takes $l_2(R^{1/\theta}) - l_1(R)$ iterations and, this time, for every iteration we pick up $\lvert \mathcal{D} \rvert$ smaller rectangles, rather than just those inside a particular column.  At this stage we are left with a large collection of rectangles with height approximately $R^{1/\theta}$ and base $m^{-l_2(R^{1/\theta})}$, which is somewhat larger.  We now iterate inside each of these rectangles until we obtain a family of rectangles each with base of length $m^{-l_1(R^{1/\theta})}$.  This takes a further $l_1(R^{1/\theta}) -l_2(R^{1/\theta})$ iterations.  We can then cover the resulting collection by small sets of diameter $R^{1/\theta}$, observing that sets formed by this last stage of iteration can be covered simultaneously, provided they are in the same column.  Therefore,  for each of the last iterations we only require a factor of $\lvert\pi\mathcal{D} \rvert$ more covering sets. Putting all these estimates together, yields
\begin{eqnarray*}
N \big( \Pi\big(Q(\mathbf{d}, R) \big) , R^{1/\theta}\big)&\asymp&  \Bigg(\prod_{l=l_2(R)+1}^{l_1(R)} C_{i_l} \Bigg) \ \Bigg( \lvert \mathcal{D} \rvert^{l_2(R^{1/\theta}) - l_1(R)} \Bigg) \ \Bigg(  \lvert\pi\mathcal{D} \rvert^{l_1(R^{1/\theta}) - l_2(R^{1/\theta})} \Bigg) \\ \\
&\leq&  \Bigg( C_{\max}\Bigg)^{l_1(R) - l_2(R)} \ \Bigg( \lvert \mathcal{D} \rvert^{l_2(R^{1/\theta}) - l_1(R)} \Bigg) \ \Bigg(  \lvert\pi\mathcal{D} \rvert^{l_1(R^{1/\theta}) - l_2(R^{1/\theta})} \Bigg) \\ \\
&\asymp&  \left(C_{\max}\right)^{\log R /\log n - \log R / \log m} \\ \\
&\,& \qquad \qquad \cdot \  \lvert \mathcal{D} \rvert^{\log R /\log m - \log R / \theta \log n}  \  \lvert\pi\mathcal{D} \rvert^{\log R / \theta\log n - \log R / \theta \log m} \\ \\
&=&  R^{\log C_{\max} / \log n + \log (\lvert \mathcal{D} \rvert/C_{\max} ) /\log m + \log(\lvert\pi\mathcal{D} \rvert/ \lvert \mathcal{D} \rvert )/ \theta \log n \, - \, \log \lvert\pi\mathcal{D} \rvert / \theta \log m}.
\end{eqnarray*}
Taking logs and dividing by $(1-1/\theta)\log R$ yields
\begin{eqnarray*}
\frac{\log N \big( \Pi\big(Q(\mathbf{d}, R) \big), R^{1/\theta}\big)}{(1-1/\theta)\log R} &\leq &  \frac{ \frac{\log C_{\max}}{\log n}\, + \, \frac{\log (\lvert \mathcal{D} \rvert/C_{\max} )}{\log m}  \, + \,  \frac{ \log(\lvert \pi\mathcal{D} \rvert/ \lvert\mathcal{D} \rvert)}{ \theta \log n}    \,  -  \,   \frac{\log \lvert\pi\mathcal{D} \rvert }{ \theta \log m}   }{1-1/\theta} \  + \ \frac{O(1)}{\log R} \\ \\
 &= &  \frac{   \bigg( \frac{\log \lvert\pi\mathcal{D} \rvert }{ \log m} \, + \,  \frac{ \log(\lvert\mathcal{D} \rvert/ \lvert \pi\mathcal{D} \rvert )}{ \log n}  \bigg) \  - \  \theta \bigg( \frac{\log (\lvert \mathcal{D} \rvert/C_{\max} )}{\log m} \, + \,  \frac{\log C_{\max}}{\log n}\bigg)  }{1-\theta} \ + \ \frac{O(1)}{\log R}.
\end{eqnarray*}
Letting $R \to 0$ yields the required upper bound, noting that these estimates are uniform in $\mathbf{d}  \in \mathcal{D}^\infty$.  The required lower bound also follows by noting that if we choose $\mathbf{d}   \in \mathcal{D}^\infty$, such that $C_{i_l} = C_{\max}$ for all $l$, then the only appearance of an inequality in the above argument is replaced by equality and observing that all of our covering estimates were optimal up to multiplicative constants.
\end{proof}

\begin{prop} \label{key2} Let   $\theta  \in [ \log m / \log n, 1)$.  Then
\[
\dim_{\mathrm{A}}^\theta F  \ = \  \frac{\log \lvert\pi\mathcal{D} \rvert }{ \log m} \ +\ \frac{\log C_{\max}}{\log n}.
\]
\end{prop}

\begin{proof}
This proof is similar to the proof of Proposition \ref{key1}.  Let 
\[
\mathbf{d} = ((i_1, j_1), (i_2, j_2), \dots)  \in \mathcal{D}^\infty
\]
and $R \in (0,1)$.  We proceed as before, but this time one obtains a family of horizontal strips with height approximately $R^{1/\theta}$ after $l_2(R^{1/\theta})$ steps, which is before the bases become smaller than $R$, since in this case $l_2(R^{1/\theta}) \leq l_1(R)$.  The effect is that the `middle term' above (concerning powers of $\lvert \mathcal{D} \rvert$) is not required.  The horizontal strips are then covered as before yielding
\begin{eqnarray*}
N \big( \Pi\big(Q(\mathbf{d}, R) \big) , R^{1/\theta}\big)&\asymp&  \Bigg(\prod_{l=l_2(R)+1}^{l_2(R^{1/\theta})} C_{i_l} \Bigg) \ \Bigg(  \lvert\pi\mathcal{D} \rvert^{l_1(R^{1/\theta}) - l_1(R)} \Bigg) \\ \\
&\leq&  \Bigg( C_{\max}\Bigg)^{l_2(R^{1/\theta}) - l_2(R)}  \  \lvert\pi\mathcal{D} \rvert^{l_1(R^{1/\theta}) - l_1(R)} \\ \\
&\asymp& \left( C_{\max}\right)^{\log R /\log n - \log R / \theta \log n}   \  \lvert\pi\mathcal{D} \rvert^{\log R /\log m - \log R / \theta \log m} \\ \\
&=&  R^{(1-1/\theta) \left(\log C_{\max} / \log n + \log\lvert \pi D \rvert /\log m\right)}.
\end{eqnarray*}
Taking logs and dividing by $(1-1/\theta) \log R$ yields
\begin{eqnarray*}
\frac{\log N \big( \Pi\big(Q(\mathbf{d}, R) \big) , R^{1/\theta}\big)}{(1-1/\theta)  \log R} &\leq& \frac{\log \lvert\pi\mathcal{D} \rvert }{ \log m} \ +\ \frac{\log C_{\max}}{\log n} \ + \ \frac{O(1)}{\log R}.
\end{eqnarray*}
Letting $R \to 0$ yields the required upper bound, noting that these estimates are uniform in $\mathbf{d}  \in \mathcal{D}^\infty$.  The required lower bound again follows by noting that if we choose $\mathbf{d}  \in \mathcal{D}^\infty$, such that $C_{i_l} = C_{\max}$ for all $l$, then the only appearance of an inequality in the above argument is replaced by equality and observing that all of our covering estimates were optimal up to multiplicative constants.
\end{proof}

We now turn to the lower spectrum, which can be handled similarly and so only the key points in the proofs are given.

\begin{prop} \label{key3} Let   $\theta  \in (0, \log m / \log n]$.  Then
\[
\dim_\mathrm{L}^\theta  F  \ = \   \frac{   \bigg( \frac{\log \lvert\pi\mathcal{D} \rvert }{ \log m} \, + \,  \frac{ \log(\lvert\mathcal{D} \rvert/ \lvert \pi\mathcal{D} \rvert )}{ \log n}  \bigg) \ -  \ \theta \bigg( \frac{\log (\lvert \mathcal{D} \rvert/C_{\min})}{\log m} \, + \,  \frac{\log C_{\min}}{\log n}\bigg)  }{1- \theta}
\]
\end{prop}

\begin{proof}
Let $\mathbf{d} = ((i_1, j_1), (i_2, j_2), \dots)  \in \mathcal{D}^\infty$ and $R \in (0,1)$.   Proceeding as above, one obtains
\begin{eqnarray*}
N \big( \Pi\big(Q(\mathbf{d}, R) \big) , R^{1/\theta}\big) &\asymp&  \Bigg(\prod_{l=l_2(R)+1}^{l_1(R)} C_{i_l} \Bigg) \ \Bigg( \lvert \mathcal{D} \rvert^{l_2(R^{1/\theta}) - l_1(R)} \Bigg) \ \Bigg(  \lvert\pi\mathcal{D} \rvert^{l_1(R^{1/\theta}) - l_2(R^{1/\theta})} \Bigg)
\end{eqnarray*}
but this time we continue by considering uniform  lower bounds:
\begin{eqnarray*}
 N \big( \Pi\big(Q(\mathbf{d}, R) \big) , R^{1/\theta}\big) &\geq&  \Bigg( C_{\min}\Bigg)^{l_1(R) - l_2(R)} \ \Bigg( \lvert \mathcal{D} \rvert^{l_2(R^{1/\theta}) - l_1(R)} \Bigg) \ \Bigg(  \lvert\pi\mathcal{D} \rvert^{l_1(R^{1/\theta}) - l_2(R^{1/\theta})} \Bigg) \\ \\
&\asymp&  \left(C_{\min}\right)^{\log R /\log n - \log R / \log m} \\ \\
&\,& \qquad \qquad \cdot \  \lvert \mathcal{D} \rvert^{\log R /\log m - \log R / \theta \log n}  \  \lvert\pi\mathcal{D} \rvert^{\log R / \theta\log n - \log R / \theta \log m} \\ \\
&=&  R^{\log C_{\min} / \log n + \log (\lvert \mathcal{D} \rvert/C_{\min} ) /\log m + \log(\lvert\pi\mathcal{D} \rvert/ \lvert \mathcal{D} \rvert )/ \theta \log n \, - \, \log \lvert\pi\mathcal{D} \rvert / \theta \log m}.
\end{eqnarray*}
Taking logs and dividing by $(1-1/\theta)\log R$ yields
\begin{eqnarray*}
\frac{\log N \big( \Pi\big(Q(\mathbf{d}, R) \big) , R^{1/\theta}\big)}{(1-1/\theta)\log R}  &\geq &  \frac{ \frac{\log C_{\min}}{\log n}\, + \, \frac{\log (\lvert \mathcal{D} \rvert/C_{\min} )}{\log m}  \, + \,  \frac{ \log(\lvert \pi\mathcal{D} \rvert/ \lvert\mathcal{D} \rvert)}{ \theta \log n}    \,  -  \,   \frac{\log \lvert\pi\mathcal{D} \rvert }{ \theta \log m}   }{1-1/\theta} \  + \ \frac{O(1)}{\log R} \\ \\
 &= &  \frac{   \bigg( \frac{\log \lvert\pi\mathcal{D} \rvert }{ \log m} \, + \,  \frac{ \log(\lvert\mathcal{D} \rvert/ \lvert \pi\mathcal{D} \rvert )}{ \log n}  \bigg) \  - \  \theta \bigg( \frac{\log (\lvert \mathcal{D} \rvert/C_{\min} )}{\log m} \, + \,  \frac{\log C_{\min}}{\log n}\bigg)  }{1-\theta}  \ + \ \frac{O(1)}{\log R}.
\end{eqnarray*}
Letting $R \to 0$ yields the required lower bound.  The required upper bound also follows by noting that if we choose $\mathbf{d}   \in \mathcal{D}^\infty$, such that $C_{i_l} = C_{\min}$ for all $l$, then the only appearance of an inequality in the above argument is replaced by equality and observing that all of our covering estimates were optimal up to multiplicative constants.
\end{proof}

\begin{prop} \label{key4} Let   $\theta  \in [ \log m / \log n, 1)$.  Then
\[
\dim_\mathrm{L}^\theta F   \ = \  \frac{\log \lvert\pi\mathcal{D} \rvert }{ \log m} \ +\ \frac{\log C_{\min}}{\log n}.
\]
\end{prop}

\begin{proof}
 Let $\mathbf{d} = ((i_1, j_1), (i_2, j_2), \dots)  \in \mathcal{D}^\infty$ and $R \in (0,1)$.  Following the strategy of the proof of Lemma \ref{key2} but this time considering uniform lower bounds we get:
\begin{eqnarray*}
N \big( \Pi\big(Q(\mathbf{d}, R) \big) , R^{1/\theta}\big) &\asymp&  \Bigg(\prod_{l=l_2(R)+1}^{l_2(R^{1/\theta})} C_{i_l} \Bigg) \ \Bigg(  \lvert\pi\mathcal{D} \rvert^{l_1(R^{1/\theta}) - l_1(R)} \Bigg) \\ \\
&\geq&  \Bigg( C_{\min}\Bigg)^{l_2(R^{1/\theta}) - l_2(R)}  \  \lvert\pi\mathcal{D} \rvert^{l_1(R^{1/\theta}) - l_1(R)} \\ \\
&\asymp& \left( C_{\min}\right)^{\log R /\log n - \log R / \theta \log n}   \  \lvert\pi\mathcal{D} \rvert^{\log R /\log m - \log R / \theta \log m} \\ \\
&=&  R^{(1-1/\theta) \left(\log C_{\min} / \log n + \log\lvert \pi D \rvert /\log m\right)}
\end{eqnarray*}
which yields
\begin{eqnarray*}
\frac{\log N \big( \Pi\big(Q(\mathbf{d}, R) \big) , R^{1/\theta}\big)}{(1-1/\theta)  \log R} &\geq& \frac{\log \lvert\pi\mathcal{D} \rvert }{ \log m} \ +\ \frac{\log C_{\min}}{\log n} \ + \ \frac{O(1)}{\log R}.
\end{eqnarray*}
Letting $R \to 0$ yields the required lower  bound.  The required upper bound again follows by noting that if we choose $\mathbf{d}  \in \mathcal{D}^\infty$, such that $C_{i_l} = C_{\min}$ for all $l$, then the only appearance of an inequality in the above argument is replaced by equality and observing that all of our covering estimates were optimal.
\end{proof}

Theorem \ref{CarpetsSpectra} follows immediately by combining Lemmas \ref{key1}, \ref{key2}, \ref{key3}, \ref{key4} and the known formulae for the box, Assouad and lower dimensions.

\subsection{Non-constant matrices and higher dimensions} \label{lalley-gatzS}

Bedford-McMullen carpets are the simplest, and probably most studied, family of self-affine sets.  However the Assouad and lower dimensions are known for considerably more general families, in particular, self-affine classes where the matrices do not need to be constant, see for example the Lalley-Gatzouras class \cite{lalley-gatz} or Bara\'nski class \cite{baranski}.  For more details on the Assouad dimensions of these carpets, we refer the reader to \cite{Mackay, Fraser}.  Also see \cite{FraserJordan} for results on the Assouad dimension of self-affine carpets with no underlying grid structure.  It would be interesting to also compute the Assouad and lower spectra for these more general families, but this seems to be considerably more complicated. We will attempt to explain this extra complication in this section via a simple heuristic, and also try to indicate what different phenomena we expect to find.

We will consider Lalley-Gatzouras carpets, which are similar to the Bedford-McMullen carpets in that they are generated by affine maps based on diagonal matrices and the matrices all contract more strongly in the vertical direction than in the horizontal direction.  Moreover, the projection onto the first coordinate is a self-similar set satisfying the open set condition and again this set plays a key role in the dimension theory of the self-affine set. The difference is that the family of matrices need not be constant and so the columns can have varying widths and within each column the rectangles can have varying heights and be distributed less rigidly.  
\begin{figure}[H] 
	\centering
	\includegraphics[width=120mm]{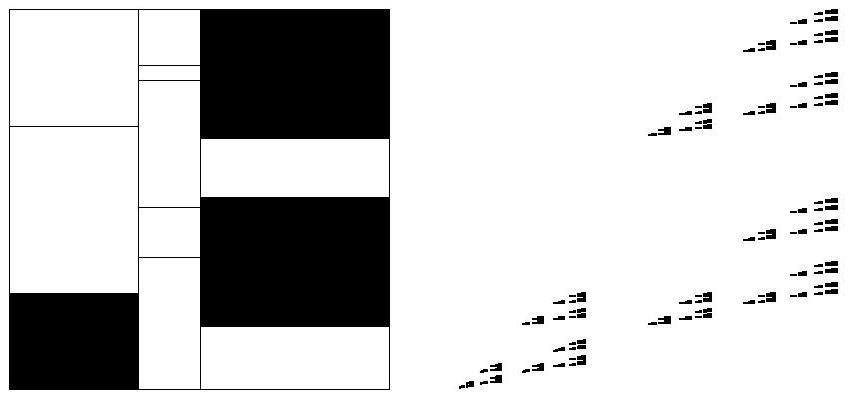}
\caption{An example of a Lalley-Gatzouras carpet.\label{carpetsLGfig}}
\end{figure}
The Assouad dimension of a Lalley-Gatzouras carpet is given by a formula similar to that used in the  Bedford-McMullen case. Indeed, it is equal to the dimension of the projection onto the first coordinate plus the largest dimension corresponding to a particular column.  This is related to `tangent sets' having a `product structure', see  \cite{Mackay}. This is also seen when computing the Assouad spectrum of Bedford-McMullen carpets.  When $\theta \geq \log m/ \log n$, the approximate squares  have a `discretised product structure' viewed at the pair of scales $R$ and $R^{1/\theta}$.  However, if $\theta < \log m/ \log n$, then one has to iterate the construction for a longer time before the height of the basic sets is roughly $R^{1/\theta}$.  This means the approximate squares lose their discrete product structure and the box dimension starts to contribute to the formula for the spectrum.  We would expect a similar phenomena to occur for Lalley-Gatzouras sets, but different columns may have a different \emph{threshold} for witnessing a product structure.  For example, suppose the column with the largest dimension, i.e., the one which contributes to the Assouad dimension consisted only of $1/4 \times 1/8$ rectangles.  This would give rise to a discrete product structure (and so the Assouad spectrum would be equal to the Assouad dimension) for $\theta> \log4/\log 8 = 2/3$, but for smaller values of $\theta$ the spectrum would decay towards the box dimension.  However, if there was another column which corresponded to a smaller dimension but which consisted only of $1/3 \times 1/9$ rectangles, then this column would allow a discrete product structure to appear on small scales for $\theta> \log3/\log 9 = 1/2$.  This would mean that the Assouad spectrum was at least equal to the dimension of the projection plus the dimension corresponding to the second column for $\theta>1/2$.  We expect that this could lead to phase transitions at $\theta =1/2$ and $\theta =2/3$;   recall that the Assouad spectrum in the Bedford-McMullen case always has precisely one phase transition.  Moreover, by carefully arranging many columns which have smaller dimension as the log-log eccentricity increases could lead to the same phenomena many times, perhaps leading to arbitrarily many phase transitions.

Computing the dimension spectra in this setting seems like a delicate and interesting problem.  Even in the simpler subcase (which we allude to here) where the matrices are constant within each column seems challenging and probably a good place to start.  The advantage of this case is that each column has a well-defined log-log eccentricity.

\begin{ques}
What are the dimension  spectra for the more general self-affine carpets considered by Lalley-Gatzouras and Bara\'nski?
\end{ques}

Another possible generalisation of our results on self-affine sets would be to consider the higher dimensional analogue of the Bedford-McMullen carpets; the so-called \emph{self-affine sponges}.  These were first considered by Kenyon and Peres  \cite{kenyon} and recently the Assouad and lower dimensions were computed by Fraser and Howroyd \cite{sponges}.  Here the maps act on $d$-dimensional space (instead of the plane) and the matrices associated to the affine maps are constantly equal to
\[
\left ( \begin{array}{cccc}
1/n_1 & 0 & \dots & 0 \\ 
0 & 1/n_2 & \dots & 0 \\ 
\vdots  &  \vdots  & \ddots & \vdots \\ 
0 & 0 & \dots & 1/n_d \\ 
\end{array} \right ) 
\]
for some fixed integers $2 \leq n_1 < n_2 < \dots < n_d$.  Again we believe that more complicated spectra are possible, with phase transitions possibly occurring at the $d-1$ (possibly distinct) ratios $\log n_k / \log n_{k+1}$ ($k=1, \dots, d-1$).  Recall that in the (non-uniform fibres) planar case precisely one phase transition occurs at the ratio $\log n_1 / \log n_{2} =\log m / \log n $.

\begin{ques}
What are the dimension spectra for self-affine sponges?
\end{ques}

\section{Self-similar and self-conformal sets with overlaps} \label{Self-similarSect}

Self-similar and self-conformal sets are another important class of fractals generated by iterated function systems.  They are simpler than the self-affine sets considered in the previous section in that the defining maps locally scale by the same amount in every direction.  If the constituent pieces making up the set do not overlap too much, then the situation is quite straightforward and it follows from standard results that the lower and Assouad dimensions coincide, see for example \cite[Corollary 2.11]{Fraser}, rendering the dimension spectra not interesting.  However, if there are overlaps in the construction then Fraser proved that the Assouad dimension can be strictly larger than the upper box dimension, but the lower dimension always equals the upper box dimension, see \cite[Section 3.1 and Corollary 2.11]{Fraser}.  Therefore,  the interesting problem is to consider the Assouad spectrum in the case of complicated overlaps.  This is not straightforward and we do not obtain precise results, but we show how to obtain non-trivial bounds in terms of the local dimensions of related Gibbs measures.  This technique was used by Fraser and Jordan to study the Assouad dimension of self-affine carpets where the self-similar set in the projection has overlaps, see \cite{FraserJordan}.

Let $U \subset \mathbb{R}^d$ be a bounded simply connected non-empty open set.  A map $S: U \to \mathbb{R}^d$ is called a \emph{conformal map} if it is differentiable and the Jacobian $S'(x)$ satisfies $|S'(x) y | =|S'(x)| | y | > 0$ for all $x \in U$ and $y \in \mathbb{R}^d \setminus \{0\}$.  Let $\{S_i\}_{i \in \mathcal{I}}$ be a finite collection of conformal maps on a common domain $U \subseteq \mathbb{R}^d$, and assume that each maps $U$ into itself.  For convenience we will extend each of the maps to the boundary of $U$ by continuity.   Furthermore, assume that each map is a bi-Lipschitz contraction, i.e. 
\[
0< \inf_{x \in U} |S'(x) |  \leq \sup_{x \in U} |S'(x) | < 1.
\]
There is a unique non-empty compact set $F$ satisfying 
\[
F \ = \ \bigcup_{i \in \mathcal{I}} S_i(F)
\]
which is called the \emph{self-conformal set} associated with the iterated function system (IFS) $\{S_i\}_{i \in \mathcal{I}}$.  See \cite[Chapter 9]{Falconer} for more details on the theory of attractors of IFSs. If for each map the Jacobian $ |S'(x) |$ is independent of $x$, then the maps $S$ are \emph{similarities} and the set $F$ is called \emph{self-similar}.  Here the situation is slightly simpler, but is subsumed into the following analysis of the self-conformal setting.

Let
\[
\mathcal{I}^\infty = \{ \mathbf{i} = (i_1, i_2, \dots) : i_l  \in \mathcal{I} \}
\]
be the set of infinite words over $\mathcal{I}$ and let $\Pi :  \mathcal{I}^\infty \to F$ be the canonical map from the symbolic space to the geometric space defined by
\[
\Pi(\mathbf{i}) \ =  \ \bigcap_{l \in \mathbb{N}} \, S_{i_1} \circ \cdots \circ  S_{i_l} \big( \overline{U}\big).
\]
Evidently $F = \Pi(\mathcal{I}^\infty)$ and we thus obtain a symbolic coding of the set $F$, but we note that some points may have many codes. Write $\mathcal{I}^* = \cup_{k \in \mathbb{N}} \mathcal{I}^k$ for the set of all finite words over $\mathcal{I}$. For $\mathbf{i} = (i_1, i_2, \dots, i_k) \in \mathcal{I}^*$ write
\[
[\mathbf{i}] \ = \ \{ \mathbf{i} \mathbf{j} : \mathbf{j} \in \mathcal I^\infty \}
\]
for the cylinder consisting of all infinite words beginning with the finite word $\mathbf{i}$ and 
\[
S_\mathbf{i} = S_{i_1}  \circ \cdots  \circ S_{i_k}.
\]
If $d \geq 2$, then it is well-known that conformal maps satisfy the \emph{bounded distortion principle}.  This means that for  $\mathbf{i} \in \mathcal{I}^k$ and $x, y \in U$
\[
|S_\mathbf{i}'(x) | \asymp |S_\mathbf{i}'(y) |
\]
with the implied constants independent of $x,y$ and $k$.  This means that for each $\mathbf{i} \in \mathcal{I}^k$, there is a constant $\text{Lip}(\mathbf{i}) \in (0,1)$ such that 
\[
|S_\mathbf{i}'(x) | \asymp |S_\mathbf{i}(\overline{U}) | \asymp \text{Lip}(\mathbf{i})
\]
where we abuse notation slightly by writing $|S_\mathbf{i}'(x) | $ for the determinant of the Jacobian  $S_\mathbf{i}'(x) $ and $|S_\mathbf{i}(\overline{U}) |$ for the diameter of the set $S_\mathbf{i}(\overline{U})$.  The constant $\text{Lip}(\mathbf{i})$ could be the determinant of the Jacobian derivative of $S_\mathbf{i}$ at the  fixed point of $S_\mathbf{i}$, or the upper (or lower) Lipschitz constant of $S_\mathbf{i}$.   For definiteness let
\[
\text{Lip}(\mathbf{i}) \ = \ \sup_{x,y \in \overline{U}} \frac{|S_\mathbf{i}(x) - S_\mathbf{i}(y) |}{|x-y|}
\]
be the upper Lipschitz constant of $S_\mathbf{i}$.  To get the bounded distortion principle to hold in $\mathbb{R}$ (which we will need), we assume  H\"older continuity of the derivatives $S_i'$.  We emphasise that this extra assumption is not required for $d \geq 2$.

The \emph{topological pressure} of the IFS is defined by
\[
P(s) \ = \ \lim_{k \to \infty} \frac{1}{k} \log \sum_{\mathbf{i} \in \mathcal{I}^k} \text{Lip}(\mathbf{i})^s
\]
where the limit is easily seen to exist by submultiplicativity of the upper Lipschitz constant for example.  Actually we have quasi-multiplicativity in the sense that for  $\mathbf{i}, \mathbf{j} \in \mathcal{I}^*$
\[
\text{Lip}(\mathbf{i}\mathbf{j}) \asymp \text{Lip}(\mathbf{i}) \text{Lip}(\mathbf{j}).
\]
The topological pressure is a continuous decreasing function of $s \geq 0$ and there is a unique zero, i.e. a unique $s\geq 0$ such that $P(s)=0$.  From now on we will fix $s$ to be the unique zero of the topological pressure.  A classical result in dimension theory states that
\[
\dim_\mathrm{L} F  =  \dim_\mathrm{H} F = \overline{\dim}_{\mathrm{B}}F \leq \min\{d,s\}
\]
and if the \emph{open set condition} is satisfied, then one has equality. This dimension formula is often referred to as \emph{Bowen's formula} goes back to Bowen \cite{Bowen} and Ruelle \cite{Ruelle}. We will be focusing on the case where the open set condition fails, but the general upper bound given by $s$ is still significant.  We recall that the Assouad dimension may exceed $s$.  Another classical result, this time from the thermodynamic formalism, states that there exists a unique  Borel probability measure $m$ on $\mathcal{I}^\infty$ (equipped with the product topology) satisfying
\[
m([\mathbf{i}]) \asymp  \text{Lip}(\mathbf{i})^s.
\]
This measure is the \emph{Gibbs measure} corresponding to the potential $\phi:\mathbf{i} \mapsto  s \log  (\text{Lip}(\mathbf{i}))$ and is a fundamental object in the  thermodynamical formalism and related problems from  statistical physics, see \cite{Ruelle, techniques}. 

Let $\mu = m \circ \Pi^{-1}$ be the push forward of $m$ onto $F$.  The measure $\mu$ is a Borel probability measure and is intimately related to the dimension theory of $F$.  The number
\[
t \ := \ \sup \left\{ t' \geq 0 \ : \  (\exists C>0) \,  (\forall x \in F) \, \mu\left( B(x,r) \right) \leq C r^{t'} \right\}
\]
will play a central role in our analysis.  This number is related to the fine structure of $\mu$ and can be expressed in terms of more familiar dimensional quantities, such as the local dimension or $L^q$-spectrum.  We record these connections to put our results in context.  The lower local dimension of a Borel probability measure $\nu$ at a point $x$ in its support is defined by
\[
\lld(\nu,x) \  = \ \liminf_{r \to 0} \frac{\log \nu\big(B(x,r) \big)}{\log r}
\]
with the upper local dimension defined similarly with $\limsup$ in place of $\liminf$.  It is clear that
\[
t \ \leq  \  \inf_{x \in F} \lld (\mu, x) \  \leq  \ \overline{\dim}_{\mathrm{B}} F \  \leq  \ \min \left\{ s, \, d\right\}
\]
The $L^q$-spectrum gives a coarse indication of the global fluctuations of a measure and are a standard tool in multifractal analysis and information theory.    For $q >0$ and $r>0$ let
\begin{eqnarray*}
M_r^q(\mu) &=& \sup \Big\{ \sum_{i} \mu(U_i)^q  \  : \   \{U_i\}_i \text{ is a centered packing of $\text{supp} \, \mu$ by balls of radius $r$ }\Big\}.
\end{eqnarray*}
The upper $L^q$-spectrum of $\mu$ is defined by
\[
\tau_{ \mu}(q) = \limsup_{r \to 0} \frac{\log M_r^q(\mu)}{\log r}.
\]
It is straightforward to see that $\tau_{ \mu}(q)$ is non-decreasing and concave in $q$.   By following the argument in \cite[Lemma 2.1]{FraserJordan}, one may prove that $t$ is actually equal to the slope of the asymptote as $q \to \infty $, i.e.,
\[
t \ = \ \inf  \left\{ t_0 \geq 0 \ : \  (\forall q>0) \,  \tau_{ \mu}(q) <  t_0q  \right\}.
\]
Moreover, if the \emph{multifractal formalism} holds for a sequence of $q \to \infty$, then
\[
t \ = \  \inf  \left\{ t_0 \geq 0 \ : \  (\forall q>0) \,  \tau_{ \mu}(q) <  t_0q  \right\} \ = \ \inf_{x \in F} \lld (\mu, x).
\]
See \cite{FengLau} for more details on the multifractal formalism.  If $F$ is self-similar, then the Gibbs measure $\mu$ is a self-similar measure and in this case we know more.  For example, it was shown by Peres and Solomyak \cite{exists} that the $L^q$-spectrum exists, and Feng \cite{Feng} proved that  the multifractal formalism holds whenever $\tau_\mu(q)$ is differentiable.  This must happen for a  sequence of $q \to \infty$ since $\tau_\mu(q)$  is non-decreasing and concave.  It seems highly likely to us that these results extend to Gibbs measure on self-conformal sets, but we omit further discussion.

We can now state our main result, which gives a non-trivial upper bound for the Assouad spectrum of $F$.

\begin{thm} \label{SSmain}
Let $F$ be a self-conformal set in $\mathbb{R}^d$ satisfying the bounded distortion principle and let $s$ and $t$ be as above.  Then
\[
\dim_{\mathrm{A}}^\theta F \ \leq \ \frac{s - t \theta}{1-\theta}. 
\]
\end{thm}

In certain situations we can say more.  For simplicity we will specialise to $d=1$ and the self-similar setting, although some results also hold in higher dimensions and for self-conformal sets.  The main result of \cite{Fraseretal} was that for a self-similar subset of the line, either the Assouad and box dimensions coincide, or the Assouad dimension is 1.  This was generalised to the self-conformal setting in \cite{sascha}; see also \cite{kaenmakiassouad2}.  The precise condition determining this dichotomy is also known to be the \emph{weak separation condition}: if the weak separation condition is satisfied, then the Assouad, Hausdorff and box dimensions of $F$ coincide, and otherwise the Assouad dimension is 1.  The weak separation condition was introduced by Lau and Ngai \cite{LauNgai} and developed by Zerner \cite{Zerner}.  For the formal definition and basic properties we refer the reader to \cite{Zerner} but, roughly speaking, it means the overlaps are `simple' in the sense that either small cylinders overlap exactly or are somewhat separated.


A famous folklore conjecture on self-similar sets is that if the semigroup generated by the defining maps  is free (i.e., there are no exact overlaps), then the Hausdorff and box dimension of the associated self-similar set should be given by $\min\{1,s\}$.  Significant progress has been made on this conjecture in recent years, in particular through the work of Hochman \cite{Hochman}.  Hochman proved that if the dimension is not given by the expected value, then there must be `super-exponential concentration of cylinders'.  The only known mechanism for such concentration is exact overlaps and if the parameters defining the maps in the IFS are algebraic then  the two phenomena are equivalent. 

Shmerkin \cite{pablo}  recently  extended Hochman's work to include the $L^q$-spectrum, which in turn allows us to obtain a precise result in an important and general setting.  For self-similar sets the measure $\mu$ is self-similar and it therefore follows from \cite[Theorem 6.6]{pablo} that the $L^q$-spectrum of $\mu$ is affine and given by $\tau_{ \mu}(q)  = s(q-1)$, provided the IFS does not have `super-exponential concentration of cylinders'.  In particular, this implies that $t=s$.  For the precise definition of   `super-exponential concentration of cylinders' we refer the reader to \cite{Hochman, pablo}.   Roughly speaking it means that sequences of construction cylinders of level $n$ cluster towards each other super-exponentially fast as $n \to \infty$.  We can glean the following corollary by combining our Theorem \ref{SSmain} with the work of Hochman \cite{Hochman} and Shmerkin \cite{pablo} and the main result of \cite{Fraseretal}.

\begin{cor} \label{SScor}
Let $F\subset \mathbb{R}$ be a self-similar set. If the IFS  does not have `super-exponential concentration of cylinders', then for all $\theta \in (0,1)$
\[
\dim_{\mathrm{A}}^\theta F \  =  \ \overline{\dim}_\mathrm{B} F.
\]
In particular, this is satisfied if the semigroup generated by the defining maps is free and  the parameters defining the maps are algebraic. If the semigroup is not free, but the weak separation property is satisfied, then for all $\theta \in (0,1)$
\[
\dim_{\mathrm{A}}^\theta F \  =  \ \overline{\dim}_\mathrm{B} F \  =  \ \dim_\mathrm{A} F.
\]
\end{cor}

 Finally we note that if a self-similar measure has a convolution structure, then one may write it as the convolution of:  a self-similar measure with no overlaps, and another measure.  Then, using the fact that taking convolutions does not lower the value of $t$, one can conclude that $t>0$. This means that we get a genuinely better upper bound than that given by \cite[Proposition 3.1]{Spectra16} in the region:
\[
\frac{s-\overline{\dim}_\mathrm{B} F }{t} < \theta < \frac{1-s}{1-t}
\] 
which is non-empty provided
\[
t >  \frac{s-\overline{\dim}_\mathrm{B} F}{1-\overline{\dim}_\mathrm{B} F}.
\]
In particular, if $s=\overline{\dim}_\mathrm{B} F$ then $t>0$ guarantees an improved bound.  For more details of the convolution argument, see \cite[Lemma 2.4]{FraserJordan}.

It would clearly be valuable to find non-trivial \emph{lower} bounds (if they exist) or to compute the spectrum explicitly for specific examples.  Two specific examples come to mind: the example from \cite[Section 3.1]{Fraser} where the overlaps are complicated by the fact that two of the maps share a fixed point; and the example considered by Bandt and Graf \cite[Section 2 (5)]{bandt} where the translations are chosen carefully to ensure that the weak separation condition fails despite all the contraction ratios being the same.

\begin{ques}
What is the Assouad spectrum for general self-similar and self-conformal sets?
\end{ques}

Given Corollary \ref{SScor} and the various  folklore conjectures on self-similar sets and measures with overlaps, we conjecture that in fact for any self-similar and self-conformal set the Assouad spectrum coincides with the upper box dimension.

\begin{figure}[H] 
	\centering
	\includegraphics[width=125mm]{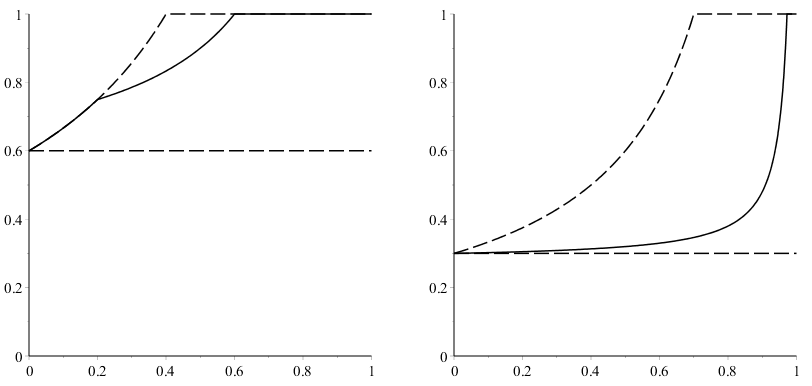}
\caption{Two plots showing improvements on the general upper bound from \cite[Proposition 3.1]{Spectra16}. In both cases the weak separation property is assumed to fail, rendering the Assouad dimension equal to 1. On the left we have chosen $t=0.5, \, \overline{\dim}_\mathrm{B} F=0.6, \, s=0.7$ and on the right we have chosen $t=0.28, \, \overline{\dim}_\mathrm{B} F= s=0.3$.   The familiar upper and lower bounds on the Assouad spectrum from \cite[Proposition 3.1]{Spectra16} are shown with dashed lines.\label{selfsimeg}}
\end{figure}

\subsection{Proof of Theorem \ref{SSmain}}

Let $x \in F$, $R \in (0,1)$ and $\theta \in (0,1)$.  For $\delta \in (0,1)$, let $\mathcal I(\delta) \subset \mathcal I^*$ be the $\delta$-\emph{stopping}, defined by
\[
\mathcal I(\delta) \ = \  \left\{ \mathbf{i} \in \mathcal I^*  \ : \ \text{Lip}(\mathbf{i}) \leq \delta <  \text{Lip}({\mathbf{i}^\dagger})   \right\}
\]
where $\mathbf{i}^\dagger$ is $\mathbf{i}$ with the final letter removed.  Let
\[
M(x, R) \ = \ \# \left\{ \mathbf{i} \in \mathcal I(R)  \ : \  S_\mathbf{i}(F) \cap B(x,R) \neq \emptyset   \right\}.
\]
All of the sets $S_\mathbf{i}(F)$ contributing to $M(x,R)$ lie completely inside the ball $B\left(x, a R \right)$ for some  fixed constant  $a>1$ which is independent of $R$ and $x$.  This follows from the bounded distortion principle.   Furthermore, the sets $S_\mathbf{i}(F)$ contributing to $M(x,R)$ all carry (symbolic) weight $m([\mathbf{i}]) \asymp \text{Lip}(\mathbf{i})^s \asymp R^s$ and so
\[
\mu \left( B\left(x, a R \right) \right)  \  \gtrsim  \   M(x,R) \, R^s.
\]
It also follows from the definition of $t$ that
\[
\mu \left( B\left(x, a R \right) \right)  \  \lesssim  \   (aR)^{t^-}  \,  \lesssim  \   R^{t^-}
\]
where, crucially, the implied constants are  independent of $x$, and therefore 
\[
M(x,R)  \  \lesssim  \   R^{t^- - s}.
\]
We wish to cover $B(x,R)$ by sets of diameter less than or equal to $R^{1/\theta}$ and to do this we iterate the construction inside each of the cylinders $S_\mathbf{i}(F)$ contributing to $M(x,R)$.  More precisely, we consider the stopping $\mathcal I(R^{1/\theta-1})$ and then decompose each set $S_\mathbf{i}(F)$ individually into the sets 
\[
\left\{ S_{\mathbf{i} \mathbf{j} }(F)  \ : \  \mathbf{j} \in  \mathcal I\left(R^{1/\theta-1} \right)  \right\}.
\]
All of these sets have diameters  $\lesssim R^{1/\theta}$ and so can themselves be covered by $\lesssim 1$ many sets of diameter less than or equal to $R^{1/\theta}$.  We claim that
\[
\left\lvert \mathcal I\left(R^{1/\theta-1} \right) \right\rvert \ \lesssim  \ R^{-s^+(1/\theta-1)}.
\]
To see this note that
\begin{eqnarray*}
\left\lvert \mathcal I\left(R^{1/\theta-1} \right) \right\rvert  R^{s^+(1/\theta-1)}  \   \asymp  \   \sum_{\mathbf{j} \in  \mathcal I\left(R^{1/\theta-1} \right)}\text{Lip}(\mathbf{j})^{s^+}    \ \leq \  \sum_{k =1}^\infty  \sum_{\mathbf{j} \in  \mathcal I^k }\text{Lip}(\mathbf{j})^{s^+}    \ < \  \infty.
\end{eqnarray*}
The fact that the final term is finite follows since
\[
\limsup_{k \to \infty} \left( \sum_{\mathbf{j} \in  \mathcal I^k }\text{Lip}(\mathbf{j})^{s^+} \right)^{1/k} \ = \  \exp(P(s^+)) < 1
\]
by the definition of $s$.  Thus
\begin{eqnarray*}
N\left(B(x,R) \cap F, \ R^{1/\theta} \right)  \ \lesssim \ M(x,R) \left\lvert \mathcal I\left(R^{1/\theta-1} \right) \right\rvert  &\lesssim&  R^{t^- - s} \  R^{-s^+(1/\theta-1)} \\ \\
&\leq&  R^{-(s^+/\theta - t^-)}
\end{eqnarray*}
and so directly from the definition and then letting $t^{-} \nearrow t$ and $s^+ \searrow s$ we conclude that
\[
\dim_{\mathrm{A}}^\theta F \ \leq \ \frac{s/\theta - t}{1/\theta - 1} \ = \   \frac{s - t \theta}{1 - \theta} 
\]
as required.

\section{Mandelbrot percolation} \label{MandelbrotSect}

Mandelbrot percolation, first appearing in the works of Mandelbrot in the 1970s as a model for intermittent turbulence \cite{Mandelbrot}, is one of the most well-studied and famous random fractal constructions and can be defined as follows.  Fix $n \geq 2$,  $d \in \mathbb{N}$ and $p \in (0,1)$.  Let $D_0=[0,1]^d$ and divide $D_0$ into $n^d$ identical closed smaller cubes by dividing each side  into $n$ equal length closed intervals. For each of the $n^d$ small cubes either `select it' with probability $p$ or `reject it' with probability $(1-p)$.  Do this independently for each cube and let the union of all the selected cubes be denoted by $D_1$.  Assuming $D_1$ is not empty, then we repeat the `dividing and selecting'  process for each of the cubes making up $D_1$ \emph{independently}.  This process is repeated inductively and, assuming non-extinction, we obtain a nested  sequence of compact sets $D_n$.  The intersection of all $D_n$ is denoted by
\[
F=\bigcap_{n=0}^{\infty}D_n
\]
and is a randomly generated compact subset of the unit cube.  The random set $F$ is the set of interest and is the known as the limit set of  Mandelbrot percolation.

\begin{figure}[H] 
	\centering
	\includegraphics[width=125mm]{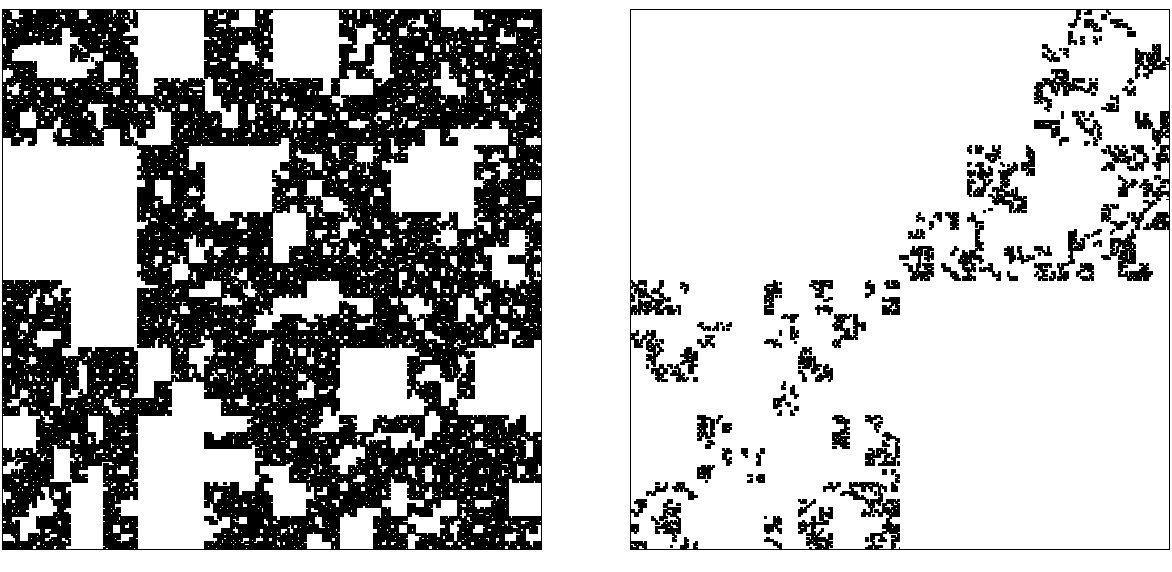}
\caption{Two examples of the limit sets of Mandelbrot percolation.  On the left $n=2$ and $p=0.9$ and on the right $n=2$ and $p=0.7$.\label{percolation}}
\end{figure}

It is well known that if $p>\frac{1}{n^d}$, then there is a positive probability of non-extinction and, conditioned on non-extinction, almost surely
\[
\dim_\mathrm{B} F=\dim_\mathrm{H} F=\frac{\log pn^d}{\log n}.
\]
It can be derived from the work of J\"arvenp\"a\"a-J\"arvenp\"a\"a-Mauldin \cite{percolationJJM} on upper and lower porosities that if $p>\frac{1}{n^d}$ then, conditioned on non-extinction, almost surely
\[
\dim_\mathrm{A} F=d \qquad \text{and} \qquad \dim_\mathrm{L} F=0 
\]
which are notably independent of $p$.  See also  Fraser-Miao-Troscheit \cite{fraserrandom}.  Our main result in this section gives the almost sure Assouad spectrum of Mandelbrot percolation.

\begin{thm}\label{MP}
Suppose $p>\frac{1}{n^d}$.  Then, conditioned on non-extinction, almost surely 
\[
\dim_\mathrm{A}^{\theta} F=\dim_\mathrm{B} F=\frac{\log pn^d}{\log n}
\]
for all $\theta\in (0,1)$.
\end{thm}

Surprisingly, the Assouad spectrum \emph{does} depend on $p$ and does not reach the Assouad dimension.  Since the spectrum is constantly equal to the (upper) box dimension, Mandelbrot percolation provides another family of examples showing that the lower bound in \cite[Proposition 3.1]{Spectra16} is sharp.

 There are many other natural and important methods for producing random fractals and it would be interesting to see if randomness tends to lead to similar behaviour more generally; i.e., the spectrum not reaching the Assouad dimension or even being constantly equal to the box dimension.  One  possible interpretation of this is as follows. Once you introduce randomness, the Assouad dimension is generically  very large because almost surely you eventually see extremal behaviour.  However, to the Assouad spectrum to be large, one needs to control (in some precise sense) how long you have to wait to see the extremal behaviour.   

\begin{ques}
What is the Assouad spectrum for natural `random fractals', for example,  the graph of fractional Brownian motion?
\end{ques}

\subsection{Branching processes}

Let $X$ be a non-negative integer-valued random variable with finite $k$-th moment for any $k\geq 0$, that is, $\mathbb{E}[X^k]<\infty$ for any $k\geq 0$. We can define a stochastic process $Y_n$ $(n\geq 0)$ by letting
$$Y_0=1$$
and for any $n\geq 1$, we define
\[
Y_{n+1}=\sum_{i=1}^{Y_n}X_i,
\]
where $X_i$ are independent identically distributed random variables with the same distribution as $X$. We adopt the convention that $\sum_{i=1}^{0}X_i=0$.  The process $Y_n$ is known as a \emph{Galton-Watson process}.  We will need the following result, which expresses the $k$th moments of $Y_n$ as an explicit polynomial.  Although the result for the first moment is well-known, we could not find a reference for the $k$th moment and so include the details for completeness.

\begin{thm}\label{BR}
For any positive integers $n,k$ we have: $$\mathbb{E}[Y_n^k]=\sum_{i=1}^{k}a_{ki}\mu^{in}$$ where $\mu=\mathbb{E}[X]$ and the coefficients  $a_{ki}$ can be computed explicitly.
\end{thm}

\begin{proof}
We will proceed by induction, first noting the well-known case when $k=1$.  For any $n\geq 0$ we have
$$\mathbb{E}[Y_{n+1}]=\mathbb{E}\left[\mathbb{E}\left[\sum_{i=1}^{Y_n}X_i\bigg\vert Y_n\right]\right]=\mathbb{E}[X]\mathbb{E}[Y_n]=\mu \mathbb{E}[Y_n]=\dots=\mu^{n+1}\mathbb{E}[Y_0]=\mu^{n+1}.$$
Suppose that for any $n$ we have
\begin{equation} \label{ind11}
\mathbb{E}[Y_n^k]=\sum_{i=1}^{k}a_{ki}\mu^{in}
\end{equation}
for $k<N$ with $N>0$ an integer.  Then
$$\mathbb{E}[Y_{n+1}^{k+1}]=\mathbb{E}\left[\mathbb{E}\left[\left(\sum_{i=1}^{Y_n}X_i\right)^{k+1}\bigg|Y_n\right]\right]$$
and we can expand the
$$\left(\sum_{i=1}^{Y_n}X_i\right)^{k+1}$$
into polynomial terms of the  following form:
$$X_{i_1}^{k_1}X_{i_2}^{k_2}\dots X_{i_s}^{k_s}$$
where $s\leq k+1$ is an integer, $i_1,\dots ,i_s$ are $s$ different integers from $1$ to $Y_n$,  and $k_1,\dots ,k_s$ are positive integers which sum to $k+1$.  If we write  $\mathbb{E}[X^m]=\mu_m$, then by independence we get
$$\mathbb{E}[X_{i_1}^{k_1}X_{i_2}^{k_2}\dots X_{i_s}^{k_s}]=\mu_{k_1}\mu_{k_2}\dots \mu_{k_s}$$ 
and so the expectations do not really depend on the  $i$-s but rather on the $k$-s.  The expansion of
\[
\mathbb{E}\left[\left(\sum_{i=1}^{Y_n}X_i\right)^{k+1}\bigg| Y_n\right]
\]
can be written as the sum of terms of the form:
$$\left({Y_n \choose s}s!\right)\mu_{k_1}\mu_{k_2}\dots \mu_{k_s}$$
where the binomial coefficients take the value $0$ when $s>Y_n$. We can arrange this sum with respect to the powers of $Y_n$ since the binomial coefficients are polynomials in $Y_n$.  We obtain an expression of the form
$$\mathbb{E}\left[\left(\sum_{i=1}^{Y_n}X_i\right)^{k+1}\bigg| Y_n\right]=\sum_{i=0}^{k+1}b_{i}Y_n^i$$
and it is clear that $b_{k+1}=\mu^{k+1}$ since there is only one way to write $k+1$ as the sum of $k+1$ positive integers, and the corresponding value of $\mu_{k_1}\dots \mu_{k_s}$ is $\mu^{k+1}$.  We then have
$$\mathbb{E}\left[\left(\sum_{i=1}^{Y_n}X_i\right)^{N}\bigg| Y_n \right]=\sum_{i=0}^{N}b_{i}Y_n^i
=\sum_{i=0}^{N-1}b_{i}Y_n^i+\mu^{N}Y_n^{N}.$$
We will write the expectation of the first term on the right side as 
\[
K(n)=\mathbb{E}\left[\sum_{i=1}^{N-1}b_{i}Y_n^i\right]=\sum_{i=1}^{N-1}b_{i}\mathbb{E}[Y_n^i].
\]
Then
\begin{eqnarray*}
\mathbb{E}[Y_{n}^N] &=& \mathbb{E}\left[\mathbb{E}\left[\left(\sum_{i=1}^{Y_{n-1}}X_i\right)^{N}\bigg|Y_{n-1}\right] \right] \\ \\
 &=& \sum_{i=0}^{N}b_{i}\mathbb{E}[Y_{n-1}^i]\\ \\
 &=& K(n-1)+\mu^{N}\mathbb{E}[Y_{n-1}^{N}] \\ \\
 &=& K(n-1)+\mu^{N}K(n-2)+\mu^{2N}K(n-3)+\dots +\mu^{N(n-1)}\mathbb{E}[Y_1^N] \\ \\
 &=& K(n-1)+\mu^{N}K(n-2)+\mu^{2N}K(n-3)+\dots +\mu^{N(n-1)}\mu_N \\ \\
 &=& \sum_{i=1}^{N-1}b_i\left(\mathbb{E}[Y_{n-1}^i]+\mu^{N}\mathbb{E}[Y_{n-2}^i]+\dots +\mu^{N(n-2)}\mathbb{E}[Y_{1}^i]\right)+\mu^{N(n-1)}\mu_N \\ \\
 &=& \sum_{i=1}^{N-1}\sum_{j=1}^{i}b_ia_{ij}\left(\mu^{j(n-1)}+\mu^{j(n-2)+N}+\dots +\mu^{j+N(n-2)}\right) + \mu^{N(n-1)}\mu_N \\ \\
 &=& \sum_{i=1}^{N-1}\sum_{j=1}^{i}b_ia_{ij}\frac{\mu^{j(n-1)}-\mu^{N(n-1)}}{1-\mu^{N-j}}+\mu^{N(n-1)}\mu_N \\ \\
 &=& \sum_{j=1}^{N-1}\sum_{i=j}^{N-1}\frac{b_ia_{ij}}{1-\mu^{N-j}}\mu^{j(n-1)}+\left(\mu_N-\sum_{i=1}^{N-1}\sum_{j=1}^{i}\frac{b_ia_{ij}}{1-\mu^{N-j}}\right)\mu^{N(n-1)}
\end{eqnarray*}

By our  inductive assumption (\ref{ind11}), all the coefficients of $\mu^{jn}$ do not depend on $n$, therefore we can of course write
 $$\mathbb{E}[Y_n^N]=\sum_{i=1}^{N}a_{Ni}\mu^{in}$$
 for $a_{Ni}$ which do not depend on $n$, completing the proof.
\end{proof}

In particular, this theorem shows that $\mathbb{E}[Y_n^N]$ can be written as a polynomial in $\mu^n$ with degree (at most) $N$ and the coefficients of the polynomial depend on $N$ but not $n$.

\subsection{Proof of Theorem \ref{MP}}

Write
\[
B=\frac{\log pn^d}{\log n}
\]
to denote the unique value of the box dimension which occurs with positive probability, and  fix $\theta \in (0,1)$.   Let $N(s)$ be the number of  cubes selected at  step $s \in \mathbb{N}$.  We call the collection of selected cubes at level $s$ the \emph{$s$-layer cubes}. They are of side length $n^{-s}$.

Let $s \in \mathbb{N}$ and $k,M>0$.  We refer to an  $s$-layer cube as a \emph{$G(s,k,M)$-cube} if there are at least $M$ cubes selected at layer $s+[k]$ which lie inside it.

The selected cubes can be considered independently and so  the probability of at least one $G(s,k,M)$-cube existing is
$$1-(1-P(s,k,M))^{N(s)},$$
where $P(s,k,M)$ is the probability that a particular $s$-layer cube is a $G(s,k,M)$-cube which, by the statistical self-similarity of the process,  is the same as the probability that $F$ itself has more than $M$  many $k$-layer cubes, i.e. $P(N(k) \geq M)$.

We can view the construction of the Mandelbrot percolation in the context of the Galton-Watson process discussed in the previous section.  In this case the branching process $Y_n$ is constructed from the Binomial random variable $X\sim B(n^d,p)$, where we note that $\mu=n^B$.  By the Markov inequality, 
$$P(s, k,M) \ = \  P(Y_{[k]}\geq M) \ \leq \ \frac{\mathbb{E}[Y_{[k]}^N]}{M^N}$$
for any $N \in \mathbb{N}$.  Therefore we have the following bound on the  probability that there exists at least one $G(s,k,M)$-cube:
$$1-(1-P(s,k,M))^{N(s)}\ \leq  \ N(s)P(s,k,M) \ \leq \  N(s)\frac{\mathbb{E}[Y_{[k]}^N]}{M^N}.$$
However, what we really want to estimate is the probability that there exists at least one $G(s,k,M)$-cube without knowledge of (conditioning on)   the value of $N(s)$.  This is easily overcome, however, and we obtain the probability
\begin{eqnarray*}
\sum_{z=0}^{n^{sd}}(1-(1-P(s,k,M))^{z})P(N(s)=z) &\leq & \sum_{z=0}^{n^{sd}}z\frac{\mathbb{E}[Y_{[k]}^N]}{M^N}P(N(s)=z) \\ \\
&\leq & \frac{\mathbb{E}[Y_{[k]}^N]}{M^N}\mathbb{E}[N(s)] \\ \\
&=& \frac{\mathbb{E}[Y_{[k]}^N]}{M^N}n^{Bs}.
\end{eqnarray*}
Therefore, for any layer $s$, the probability that there exists at least one $G(s,s/\theta-s,n^{B_1(1/\theta-1)})$-cube can be bounded from above by $$n^{Bs}\frac{\mathbb{E}[Y_{[s(1/\theta-1)]}^N]}{n^{NB_1s(1/\theta-1)}}.$$
Here $B_1>0$ is a non-specified constant.  For large enough $s$, we have from Theorem \ref{BR} that $$\mathbb{E}[Y_{[s(1/\theta-1)]}^N]$$ is a polynomial in $n^{Bs(1/\theta-1)}$ with degree at most $N$.  Therefore we can bound the whole term  from above by $$Ln^{NBs(1/\theta-1)}$$ where $L$ is a positive number depending only on $N$.  We have
$$n^{Bs}\frac{\mathbb{E}[Y_{[s(1/\theta-1)]}^N]}{n^{NB_1s(1/\theta-1)}} \ \leq \  Ln^{s(B+BN(1/\theta-1)-NB_1(1/\theta-1))}$$
which means that if
\begin{eqnarray} \label{BC111}
B+BN(1/\theta-1)-NB_1(1/\theta-1)<0
\end{eqnarray}
then
\[
\sum_{s =1}^\infty  n^{Bs}\frac{\mathbb{E}[Y_{[s(1/\theta-1)]}^N]}{n^{NB_1s(1/\theta-1)}} \ < \ \infty.
\]
The Borel-Cantelli Lemma then implies that almost surely there are only finitely many layers $s$ such that there exists at least one $G(s,s/\theta-s,n^{B_1(1/\theta-1)})$-cube.  This implies that almost surely $\dim_\mathrm{A}^{\theta} F \leq B_1^+$, provided (\ref{BC111}) holds.  Therefore, almost surely, we have
$$\dim_\mathrm{A}^{\theta} F\leq \frac{B\theta}{N(1-\theta)}+B$$
and since $N$ can be chosen arbitrarily large, it follows that  almost surely, $\dim_\mathrm{A}^{\theta} F\leq B$. The reverse inequality holds almost surely, conditioned on non-extinction, due to the general bounds from \cite[Proposition 3.1]{Spectra16} and the classical dimension result for box dimension.  Therefore, for our fixed choice of $\theta$,  almost surely, conditioned on non-extinction,  we have
\[
\dim_\mathrm{A}^{\theta} F= B.
\]
Finally, since $\theta \in (0,1)$ was arbitrary we can obtain this result simultaneously for all $\theta \in (0,1) \cap \mathbb{Q}$ and since we know the Assouad spectrum is continuous in $\theta$ (\cite[Corollary 3.5]{Spectra16}), this is sufficient to move the quantifier on $\theta$ inside the almost sure result, which proves the theorem.

\section{Moran constructions} \label{MoranSection}

Moran constructions are an important and well-studied class of fractals which can be thought of in a similar way to attractors of IFSs but where the rigidity between construction levels has been relaxed in a number of different ways.  Similar to IFS attractors, one defines a set as the intersection of a nested sequence of compact sets made up of increasingly small construction cylinders.  However, the difference is that the relative position, size and number  of `child' cylinders within their `parent' can vary.  This leads to a richer theory, less rigid construction rules and, as we shall see, more complicated dimension results.

A natural way to code the nested construction cylinders is via an infinite rooted tree.  Let $T$ be such a tree with root vertex $v_0$ on level 0 and such that every vertex on level $n \geq 0$ has at least one `child' vertex at level $n+1$.  This guarantees that there are no `degenerate' paths, i.e., paths starting from $v_0$ and terminating at some finite level.  In general, the `level' of a vertex is defined uniquely as the graph distance of the vertex to the root $v_0$. Let $E$ be the set of  edges, $V$ be the set of vertices, $l(v) \in \mathbb{N}$ be the level of vertex $v \in V$, and $V(k)$  be the set of level $k$ vertices.

Let $M(v_0) = M_0 \subseteq \mathbb{R}^d$ be a compact convex set which is equal to the closure of its interior, for example we could take $M_0$ to be a ball or a cube.  For each $v \in V \setminus \{v_0\}$ we assign $c(v) \in (0,1)$ and a set $M(v) \subseteq \mathbb{R}^d$ with diameter $\lvert M(v) \rvert$ such that
\begin{enumerate}
\item For all $v \in V$, $M(v)$ is similar to $M_0$, i.e., the image of $M_0$ under a similarity map.
\item If $v' \in V$ is a child of $v \in V$, then $M(v')  \subseteq M(v)$
\item  If $v' \in V$ is a child of $v \in V$, then $\lvert M(v') \rvert = c(v') \lvert M(v) \rvert $
\item If $v', v'' \in V$ are distinct children of $v \in V$, then the interiors of $M(v')$ and $M(v'')$ are disjoint
\item There is a uniform constant $c>0$ such that $c(v) \geq c$ for all $v \in  \setminus \{v_0\}$.
\end{enumerate}

By construction the sequence $\cup_{v \in V(k)}  M(v)$ is a decreasing sequence of non-empty compact sets and so we may define 
\[
M \ = \ \bigcap_{k=1}^\infty \bigcup_{v \in V(k)}  M(v)
\]
which is itself a non-empty compact, typically fractal, set and is the main object of study in this section. Such sets are often called \emph{Moran sets} or \emph{Moran constructions}.  For convenience, we make a further homogeneity assumption that the contraction constants $c(v)$ are uniform across a given level. More precisely, we assume that for each $k \in \mathbb{N}$, there is a constant $c(k) \in (0,1)$ such that for all $v \in V(k)$ we have $c(v) = c(k)$.  In this situation we call the set $M$ a \emph{homogeneous Moran set}.

Before we can state our main result, we need some more notation.  For $v \in V(k)$ and $l \geq 1$, let $N(v,l)$  be the number of level $k+l$ children of $v$, i.e., the number of vertices one can get to from $v$ by going a further $l$ levels down the tree.  Given $\theta \in (0,1)$ and $k \in \mathbb{N}$ let $l(\theta, k) \in \mathbb{N}$ be defined by
\[
l(\theta, k) \ = \ \max \left\{ l \in \mathbb{N} \ : \ \prod_{i=1}^l  c(i) \geq \prod_{i=1}^k  c(i) ^{1/\theta} \right\}
\]
and observe that $l(\theta,k) \geq k$.

\begin{thm} \label{moranthm}
Let $M$ be a homogeneous Moran set.  Then
\[
\dim_{\mathrm{A}}^\theta M  \ = \  \limsup_{k \to \infty}  \  \max_{v \in V(k)} \, \frac{\log  N(v, l(\theta,k)-k) }{(1-1/\theta)\log  \prod_{i=1}^k  c(i)} 
\]
and
\[
\dim_\mathrm{L}^\theta  M  \ = \  \liminf_{k \to \infty}   \  \min_{v \in V(k)} \, \frac{\log  N(v, l(\theta,k)-k) }{(1-1/\theta)\log  \prod_{i=1}^k  c(i)} .
\]
\end{thm}

Since the dimension spectra are local quantities, it is instructive to make the further simplifying assumption that the growth is also uniform across level. More precisely, suppose that for each $k \in \mathbb{N}$, there is a constant $N(k) \geq 1$ such that for all $v \in V(k)$ we have $N(v,1) = N(k)$.  In this setting we call $M$ a \emph{uniformly homogeneous Moran set}. This makes describing examples with particular properties more straightforward as one only has to specify the sequences $c(k)$ and $N(k)$. 

\begin{cor} \label{morancor}
Let $M$ be a uniformly homogeneous Moran set.  Then
\[
\dim_{\mathrm{A}}^\theta M  \ = \  \limsup_{k \to \infty}   \, \frac{\log  \prod_{i=k+1}^{l(\theta, k)}  N(i) }{(1-1/\theta)\log  \prod_{i=1}^k  c(i)} 
\]
and
\[
\dim_\mathrm{L}^\theta  M  \ = \  \liminf_{k \to \infty}  \, \frac{\log  \prod_{i=k+1}^{l(\theta, k)}  N(i) }{(1-1/\theta)\log  \prod_{i=1}^k  c(i)}.
\]
Moreover, if $c(i) = c \in (0,1)$ for all $i \geq 1$, then
\[
\dim_{\mathrm{A}}^\theta M  \ = \  \frac{\theta }{(\theta - 1) \log  c} \limsup_{k \to \infty}   \, \frac{1}{k} \log  \prod_{i=k+1}^{\lfloor k/\theta \rfloor}  N(i)
\]
and
\[
\dim_\mathrm{L}^\theta  M  \ = \  \frac{\theta }{(\theta - 1) \log  c} \liminf_{k \to \infty}   \, \frac{1}{k} \log  \prod_{i=k+1}^{\lfloor k/\theta \rfloor}  N(i).
\]
\end{cor}

\subsection{Proof of Theorem \ref{moranthm}}

Let $R, \theta \in (0,1)$ and define $k \in \mathbb{N}$ by
\[
k \ = \ \max \left\{ l \in \mathbb{N} \ : \ \prod_{i=1}^l  c(i) \geq R \right\}.
\]
Observe that level $k$ cylinders $M(v)$ are images of $M_0$ scaled down by similarities with contraction ratio $\asymp R$, where the implied constant depends only on the constant $c$ appearing in condition (5) above.  This guarantees that any $R$-ball centered in $M$ will intersect $\lesssim 1$ many level $k$ cylinders.  This  uses the fact that each level $k$ cylinder contains a ball with radius $\gtrsim R$ where the implied constant depends only on $M_0$ and by construction these balls may be taken to be disjoint.  The conclusion then follows from the doubling property of Euclidean space.   Furthermore, a particular level $k$ cylinder $M(v)$ may be broken down into precisely $N(v, l(\theta,k)-k)$ many level $l(\theta,k)$ cylinders. Each of these cylinders is a scaled down copy of $M_0$ by a factor of
\[
\prod_{i=1}^{ l(\theta,k)} c(i) \asymp \prod_{i=1}^k  c(i) ^{1/\theta} \asymp R^{1/\theta}
  \]
  and may therefore be covered by $\lesssim 1$ many open balls of diameter $R^{1/\theta}$.   Thus we may cover any $R$-ball with
\[
\lesssim \max_{v \in V(k)} N(v, l(\theta,k)-k)
\]
many $R^{1/\theta}$-balls.  This shows that
\begin{eqnarray*}
\dim_{\mathrm{A}}^\theta F   \  = \  \limsup_{R \to 0} \  \sup_{x \in M} \, \frac{\log  N\big(B(x,R) \cap M,R^{1/\theta}\big) }{(1-1/\theta)\log R} &\leq& \limsup_{k \to \infty} \   \frac{\log \left( \max_{v \in V(k)} N(v, l(\theta,k)-k) \right)}{(1-1/\theta)\log \left( \prod_{i=1}^k c(i) \right)}
\end{eqnarray*}
which is the desired upper bound. To prove the lower bound, choose a sequence of $ k \to \infty$ and corresponding maximising $v \in V(k)$ realising the $\limsup$ above.  Observe that the cover constructed above is optimal up to a constant since any ball of diameter $R^{1/\theta}$ used in the cover can cover at most $\lesssim 1$ many level $l(\theta,k)$ cylinders.  This again uses the fact that each level $l(\theta,k)$ cylinder contains a ball with radius $ \gtrsim R^{1/\theta}$ where the implied constant depends only on $M_0$ and by construction these balls may be taken to be disjoint. Therefore choosing $x \in M(v) \cap M$ we have
\[
 N\big(B(x,R) \cap M,R^{1/\theta}\big) \ \gtrsim \ N(v, l(\theta,k)-k)  \ = \ \max_{v \in V(k)} N(v, l(\theta,k)-k).
\]
This proves the lower bound and completes the proof for the Assouad spectrum.  The analogous result for the lower spectrum is proved similarly and we omit the details.

\subsection{Examples and applications}  \label{MoranExamples}

Let $M_0 = [0,1]$,  $c(k) \equiv 1/2$ and $N(k) \in \{1,2\}$ for all $v \in V$.  For example, we can think of the construction cylinders $M(v)$ as just the closed dyadic intervals.  Let
\[
T(k, \theta) = \#\left\{i=k+1, \dots, \lfloor k/\theta \rfloor \ : \ N(i)= 2 \right\}.
\]
It follows from Corollary \ref{morancor} that
\[
\dim_{\mathrm{A}}^\theta M  \ = \  \limsup_{k \to \infty}   \, \frac{T(k, \theta)}{k/\theta-k}
\]
and
\[
\dim_\mathrm{L}^\theta  M  \ = \   \liminf_{k \to \infty} \, \frac{T(k, \theta)}{k/\theta-k}.
\]
This gives us an easy way of building interesting examples simply by specifying the sequence $N(i)$.

Note that the upper and lower box dimensions of $M$ are given by the, more familiar, upper and lower \emph{asymptotic densities}:
 \[
\overline{\dim}_{\mathrm{B}}M  \ = \  \limsup_{k \to \infty}   \, \frac{\#\left\{i=1, \dots, k \ : \ N(i)= 2 \right\}}{k} 
\]
and
 \[
\underline{\dim}_\mathrm{B} M  \ = \  \liminf_{k \to \infty}   \, \frac{\#\left\{i=1, \dots, k \ : \ N(i)= 2 \right\}}{k} 
\]
and the Assouad and lower dimensions are given by the upper and lower \emph{Banach densities}:
 \[
\dim_\mathrm{A} M  \ = \  \limsup_{(k-l) \to \infty}   \, \frac{\#\left\{i=l, \dots, k \ : \ N(i)= 2 \right\}}{k-l+1} 
\]
and
 \[
\dim_\mathrm{L} M  \ = \  \liminf_{(k-l) \to \infty}   \, \frac{\#\left\{i=l, \dots, k \ : \ N(i)= 2 \right\}}{k-l+1}.
\]
The density functions we study are clearly related to the asymptotic and Banach densities but seem not to be so well-studied.  Generally, for $\lambda>1$ one might define the \emph{upper $\lambda$-tail density} of a set $X \subset \mathbb{N}$ as
\[
\overline{D}(X, \lambda) \ = \  \limsup_{k \to \infty}   \, \frac{\# X \cap [k, \lambda k]}{\lambda k-k}
\]
and the \emph{lower $\lambda$-tail density} as
\[
\underline{D}(X, \lambda) \ = \  \liminf_{k \to \infty}   \, \frac{\# X \cap [k, \lambda k]}{\lambda k-k}.
\]
In the same way that our spectra give finer information regarding the geometric scaling and homogeneity of the fractal set in question, the $\lambda$-tail densities would give finer information (via upper and lower spectra) regarding the density of the set $X$ within $\mathbb{N}$.  Due to the relationships between the $\lambda$-tail densities and our dimension spectra, we can deduce several basic properties of the tail densities immediately from the basic properties of the dimension spectra.

\begin{prop} \label{taildensity}
Let $X \subseteq \mathbb{N}$  and write  $\underline{D}(X)$, $\overline{D}(X)$, $\underline{B}(X)$, $\overline{B}(X)$, for the lower and upper asymptotic densities and the lower and upper Banach densities respectively.
\begin{enumerate}
\item The upper and lower $\lambda$-tail densities are continuous as functions of $\lambda>1$
\item For any $\lambda>1$ we have
\[
\overline{D}(X) \ \leq \  \overline{D}(X, \lambda)  \ \leq \    \frac{\lambda \, \overline{D}(X)}{\lambda - 1}  \,  \wedge \,  \overline{B}(X)
\]
\item For any $\lambda>1$ we have
\[
 \frac{\lambda \, \underline{D}(X)-1}{\lambda - 1} \,  \vee \,  \underline{B}(X)\ \leq \  \underline{D}(X, \lambda)  \ \leq \   \underline{D}(X) 
\]
\end{enumerate}
\end{prop}

\begin{proof}
(1) and (2) follow immediately from \cite[Corollary 3.5, Corollary 3.7, Proposition 3.1]{Spectra16} and setting $\lambda=1/\theta$.  If either were false, there would exist a set $X$ exhibiting their falsehood and choosing the sequence $N(i)$ to be 2 on elements indexed by $X$ and 1 otherwise would yield a Moran construction with spectra contradicting the basic results we have derived.   (3) follows by observing that
\[
\underline{D}(X, \lambda) \ = \ 1 - \overline{D}(\mathbb{N} \setminus X, \lambda)
\]
and applying (2).
\end{proof}

Similar to \cite[Corollary 3.3]{Spectra16}, we see that if the upper asymptotic density is 0 then the upper $\lambda$-tail density is identically 0 and, furthermore, if the lower asymptotic density is 1 then the lower $\lambda$-tail density is identically 1.

Returning to the issue of examples, first we give an example illustrating the sharpness of \cite[Proposition 3.1, 3.10]{Spectra16}. 

Let $N(i)$ be a sequence of $1$s and $2$s where the $2$s are placed with $\lambda$-tail density equal to 1/2 for all $\lambda>1$, but that there are arbitrarily long sequences of 1s and 2s (thus making the upper and lower Banach densities  1 and 0 respectively).  One way to build such a sequence would be to begin with the sequence $1,2,1,2,1,2, \dots$ and then for each $n \in \mathbb{N}$  insert a run of $n$ 1s followed by $n$ 2s starting at position $f(n)$ where $f(n)$ is some sequence which grows very fast, for example $n^{(n^n)}$. It follows that
\[
0=\dim_\mathrm{L} M<\dim_\mathrm{L}^\theta M=\dim_\mathrm{B} M=\dim_\mathrm{A}^\theta M=1/2<\dim_\mathrm{A} M=1
\]
and so this gives a simple demonstration that the lower (respectively upper) bounds from \cite[Proposition 3.1, 3.10]{Spectra16} are sharp.

In a different direction, if we let $M^*$ be the set produced by following the construction of $M$ but with the placement of 1s and 2s reversed, then we get
\[
\dim_{\mathrm{A}}^\theta M^*  \ = \  1-\dim_\mathrm{L}^\theta  M
\]
and
\[
\dim_\mathrm{L}^\theta  M^*  \ = \  1- \dim_{\mathrm{A}}^\theta M
\]
with similar formulae relating the Assouad and lower dimensions and the upper and lower box dimensions. This is a very simple observation, but it allows one to build examples with lower spectrum equal to
\[
\dim_\mathrm{L}^\theta  M^* \ = \ \frac{\underline{\dim}_\mathrm{B} M^*-\theta }{1-\theta} \ \vee  \   \dim_\mathrm{L} M^*
\]
which has not been seen yet and tends not to turn up as naturally as the analogous formula for the Assouad spectrum.  Also it means one can produce examples simultaneously for the Assouad and lower spectrum whilst only worrying about the upper densities.

We end this section with a simple recipe for building Moran constructions with interesting spectra, but we leave the details of the calculations to the interested reader.  Let $t \in [0,1]$ and $\lambda>1$ and choose the sequence $N(i)$ as follows.  Every entry is a 1, apart from for a sequence of blocks $f(k), \dots , [\lambda f( k)]$ where $f(k)$ is a sequence of integers which grows very quickly.  Within each block $f(k), \dots , [\lambda f(k)]$, distribute 2s with density $t$ and as uniformly as possible, i.e. choose $[([\lambda f(k)]-f(k)+1) t ]$ many of the $N(i)$ to be 2, rather than 1.  Simple calculations reveal that for the corresponding Moran construction $M$ we have
\[
\underline{\dim}_\mathrm{B} M =  \dim_\mathrm{L} M =  0, \qquad \qquad   \overline{\dim}_\mathrm{B} M = \frac{t}{\lambda}(\lambda-1), \qquad  \qquad \dim_\mathrm{A} M = t
\]
and for all $\theta \in (0,1)$
\[
 \dim^\theta_\mathrm{A} M = \frac{\frac{t}{\lambda}(\lambda-1)}{1-\theta} \wedge t.
\]
This simple construction allows us to build a variety of Moran constructions with interesting spectra.  For example, we may vary the parameters $t,\lambda$ above to create finitely many Moran constructions, $M_i$, with different spectra of the above form.   We can then  build another Moran  construction, $M$, which is the disjoint union of scaled down copies of the $M_i$.  This can be done by choosing as many first level vertices as there are sets $M_i$ and then mimicking the constructions of each $M_i$ independently within each first level child.  Finally, using the fact that the Assouad spectrum is stable under taking finite unions, see \cite[Proposition 4.1]{Spectra16}, the Assouad spectrum of $M$ is given by the maximum of the Assouad spectra of the $M_i$.  
 Finally, by `inverting' the construction of each $M_i$ to $M_i^*$ as above, we get a corresponding example with lower spectrum mirroring the Assouad spectrum of the initial example.  Observe that this time the lower spectrum of $M^*$ is given by \emph{minimum} of the lower  spectra of the $M_i^*$, see \cite[Proposition 4.1]{Spectra16}. Some examples of this type are given in Figure \ref{moranfig}, with the details below.  In particular, this recipe allows us to produce examples with the following properties, which we have not observed in any of our previous examples:
\begin{enumerate}
\item An arbitrarily large (but finite) number of phase transitions.
\item An arbitrarily large (but finite) number of disjoint intervals where the spectrum is constant (with different constants).
\item As $\theta$ approaches 1, the spectra approach intermediate values, not equal to any of the familiar dimensions.
\end{enumerate}

\begin{figure}[H]
	\centering
	\includegraphics[width=145mm]{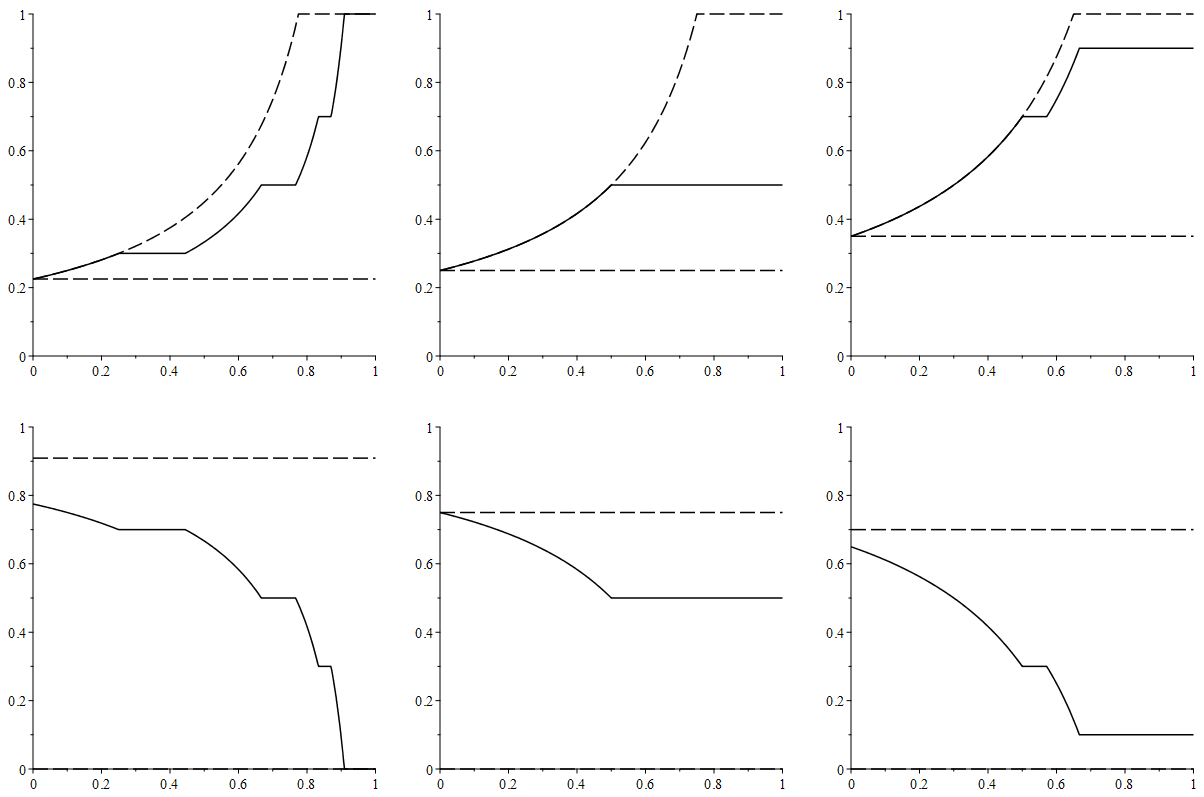}
\caption{Top row:  the Assouad spectra of three different homogeneous Moran constructions.  The general  bounds from  \cite[Proposition 3.1]{Spectra16} are shown as dashed lines.   Bottom row:  the lower spectra of the corresponding Moran construction $M^*$.  The general  bounds from  \cite[Proposition 3.10]{Spectra16} are shown as dashed lines.\label{moranfig}}
\end{figure}

For completeness we provide the precise construction data for the above examples.  The first example is produced with four basic Moran constructions with: $t=1$ and $\lambda=1.1$, $t=0.7$ and $\lambda=1.2$, $t=0.5$ and $ \lambda=1.5$, and $t=0.3$ and $ \lambda=4$, respectively.  The second example is produced with two basic Moran constructions.  The first has  $t=0.5$ and $ \lambda=2$ and the second one is produced by a sequence $N(i)$ which has all 1s apart from arbitrarily long runs of 2s inserted at positions which increase very rapidly.  This guarantees that the upper box dimension is 0, but the Assouad dimension is 1.  The third example  is produced with three basic Moran constructions.  The first two have $t=0.7$ and $\lambda=2$, and   $t=0.9$ and $ \lambda=1.5$, respectively, and the third one is again the construction with Assouad dimension 1 and upper box dimension 0.

\bibliographystyle{amsalpha}
\bibliography{}

\printindex

\end{document}